\documentclass[a4paper]{article}

\usepackage[utf8]{inputenc}
\usepackage{amsmath}
\usepackage{amsfonts}
\usepackage{amssymb}
\usepackage{graphicx}
\usepackage{longtable}
\usepackage{amscd,enumerate}
\usepackage{setspace}
\usepackage{color}
\usepackage{cite}
\usepackage{accents}
\usepackage{mathtools}
\usepackage[overload]{empheq}
\usepackage{etoolbox}
\usepackage[colorlinks]{hyperref}

\newtheorem{theorem}{Theorem}[section]

\newtheorem{assumption}[theorem]{Assumption}

\newtheorem{convention}[theorem]{Convention}

\newtheorem{definition}[theorem]{Definition}

\newtheorem{lemma}[theorem]{Lemma}

\newtheorem{problem}[theorem]{Problem}
\newtheorem{proposition}[theorem]{Proposition}
\newtheorem{remark}[theorem]{Remark}

\makeatletter\@addtoreset{equation}{section}\makeatother
\newenvironment{proof}[1][Proof]{\noindent\textbf{#1.} }{\ \rule{0.5em}{0.5em}}

\evensidemargin=\oddsidemargin
\newdimen\dummy
\dummy=\oddsidemargin
\newcommand\intervalCC[1]{\left[#1\right]}
\newcommand\intervalCO[1]{\left[#1\right[}

\newcommand\intervalOO[1]{\left]#1\right[}
\newcommand\identity{Id}
\newcommand\de{\ensuremath{\mathop{}\!\mathrm{d}}}
\newcommand\transpose[1]{\ensuremath{#1^{\top}}}
\newcommand\imaginary{\ensuremath{\mathrm{i}}}
\newcommand\conjugate[1]{\ensuremath{\overline{#1}}}

\newcommand\vectorValued[1]{\ensuremath{\boldsymbol{#1}}}
\newcommand\timeOperator[1]{\ensuremath{\mathcal{#1}}}
\newcommand\laplaceOperator[1]{\ensuremath{\mathsf{#1}}}
\newcommand\timeFunction[1]{\ensuremath{#1}}
\newcommand\laplaceFunction[1]{\ensuremath{\hat{#1}}}

\newcommand\normalN{\ensuremath{\operatorname n}}
\newcommand\swapifbranches[3]{#1{#3}{#2}}
\makeatletter
\MHInternalSyntaxOn
\patchcmd{\DeclarePairedDelimiter}{\@ifstar}{\swapifbranches\@ifstar}{}{}
\MHInternalSyntaxOff
\makeatother

\DeclarePairedDelimiter\abs{\lvert}{\rvert}
\DeclarePairedDelimiter\norm{\lVert}{\rVert}
\DeclarePairedDelimiter\restrict{.}{\vert}

\DeclareMathOperator\real{Re}
\DeclareMathOperator\diag{diag}
\DeclareMathOperator\divergence{div}
\newcommand\gradient{\ensuremath{\nabla}}

\DeclareMathOperator\dataDirichlet{D}
\DeclareMathOperator\dataNeumann{N}
\DeclareMathOperator\dataImpedance{I}
\DeclareMathOperator\dataJump{J}
\DeclareMathOperator\e{e}

\newcommand\trace[1]{\ensuremath{\vectorValued \gamma_{#1}}}
\newcommand\traceFull{\ensuremath{\trace{\operatorname C}}}
\newcommand\traceScaled[1]{\ensuremath{\trace{#1}(s)}}
\newcommand\traceFullScaled{\ensuremath{\traceFull(s)}}
\newcommand\traceD[1]{\ensuremath{\gamma_{\dataDirichlet, #1}}}
\newcommand\traceN[1]{\ensuremath{\gamma_{\dataNeumann, #1}}}
\newcommand\traceDFull{\ensuremath{\gamma_{\dataDirichlet}}}
\newcommand\traceNFull{\ensuremath{\gamma_{\dataNeumann}}}
\newcommand\traceNormal[1]{\ensuremath{\gamma_{\normalN, #1}}}
\newcommand\complementTrace[1]{\ensuremath{#1^c}}

\newcommand\imp[1]{\ensuremath{#1^{\operatorname{imp}}}}
\newcommand\zero[1]{\ensuremath{#1^0}}

\newcommand\aMix{\ensuremath{a^{\operatorname{mix}}}}
\newcommand\lMix{\ensuremath{\ell^{\operatorname{mix}}}}
\newcommand\aImp{\ensuremath{\imp{a}}}
\newcommand\lImp{\ensuremath{\ell^{\operatorname{imp}}}}
\newcommand\aZero{\ensuremath{\zero{a}}}
\newcommand\lZero{\ensuremath{\zero{\ell}}}
\newcommand\uZero{\ensuremath{\zero{u}}}

\newcommand\XTraces[1]{\ensuremath{\vectorValued X_{#1}}}
\newcommand\XSingle{\ensuremath{\vectorValued X^{\operatorname{single}}}}
\newcommand\XSingleZero{\ensuremath{\zero{\vectorValued X}}}
\newcommand\XMulti{\ensuremath{\vectorValued X^{\operatorname{mult}}}}

\begin{document}
\title{A Stable Boundary Integral Formulation of an Acoustic Wave Transmission
Problem with Mixed Boundary Conditions}
\author{S. Eberle\thanks{(eberle@math.uni-frankfurt.de), Institut f\"{u}r Mathematik,
Goethe-Universit\"{a}t Frankfurt am Main, Robert-Mayer-Str. 10, 60325
Frankfurt am Main, Germany}
\and F. Florian\thanks{(francesco.florian@uzh.ch), Institut f\"{u}r Mathematik,
Universit\"{a}t Z\"{u}rich, Winterthurerstr 190, CH-8057 Z\"{u}rich,
Switzerland}
\and R. Hiptmair\thanks{(ralf.hiptmair@sam.math.ethz.ch), Seminar für Angewandte Mathematik, ETH Zürich, Rämistrasse 101, CH-8092 Zürich, Switzerland}
\and S.A. Sauter\thanks{(stas@math.uzh.ch), Institut f\"{u}r Mathematik,
Universit\"{a}t Z\"{u}rich, Winterthurerstr 190, CH-8057 Z\"{u}rich,
Switzerland}}
\maketitle

\begin{abstract}
In this paper, we consider an acoustic wave transmission problem with mixed
boundary conditions of Dirichlet, Neumann, and impedance type. We will derive
a formulation as a \textit{direct}, \textit{space-time retarded boundary
integral equation}, where both Cauchy data are kept as unknowns on the
impedance part of the boundary. This requires the definition of single-trace
spaces which incorporate homogeneous Dirichlet and Neumann conditions on the
corresponding parts on the boundary. We prove the continuity and coercivity of
the formulation by employing the technique of operational calculus in the
Laplace domain.

\end{abstract}

\textbf{Keywords}: acoustic wave equation, transmission problem, impedance
boundary condition, retarded potentials, convolution quadrature

\section{Introduction}
\subsection{Transmission Problems}
In physics and engineering there are many important applications where it is
essential to obtain information on material properties inside (large) solid
objects, e.g., the detection of oil reservoirs, the investigation of the
interior of rocks and soil to understand its stability properties, or the
assessment of the ice volume in glaciers to name just a few of them. For this
purpose, typically, a wave is sent into the solid. Then the scattered wave is
recorded and used to solve the governing mathematical equations for the
quantity of interest.

Our goal is to employ the method of integral equation to reformulate the
scalar wave equation as a system of space-time boundary integral equations;
standard references on this topic include \cite{Stratton_book,friedman,bambduong,sayas_book}. The \textit{Cauchy
data}, i.e., Dirichlet and Neumann traces on the boundary, of the wave (or
boundary densities if an \textit{indirect formulation} is employed) is
determined as the solution of a system of retarded potential integral
equations (RPIE). To investigate well-posedness we employ the Laplace
transform and prove continuity and coercivity with respect to the frequency
variable. These techniques in the context of numerical analysis go back to the
pioneering works \cite{bambduong,lub1,lub2,duong03}; a
monograph on this topic is \cite{sayas_book} and some further developments can
be found, e.g., in \cite{LaSa2009} and \cite{banjai_coupling}.

We emphasize that the derivation of coercive and continuous integral equations
in the Laplace domain is key for their discretization by \textit{convolution
quadrature}. However, here we will focus on the continuous formulation and
prove its well-posedness.

\begin{figure}
  \begin{center}
    \includegraphics[scale=.3]{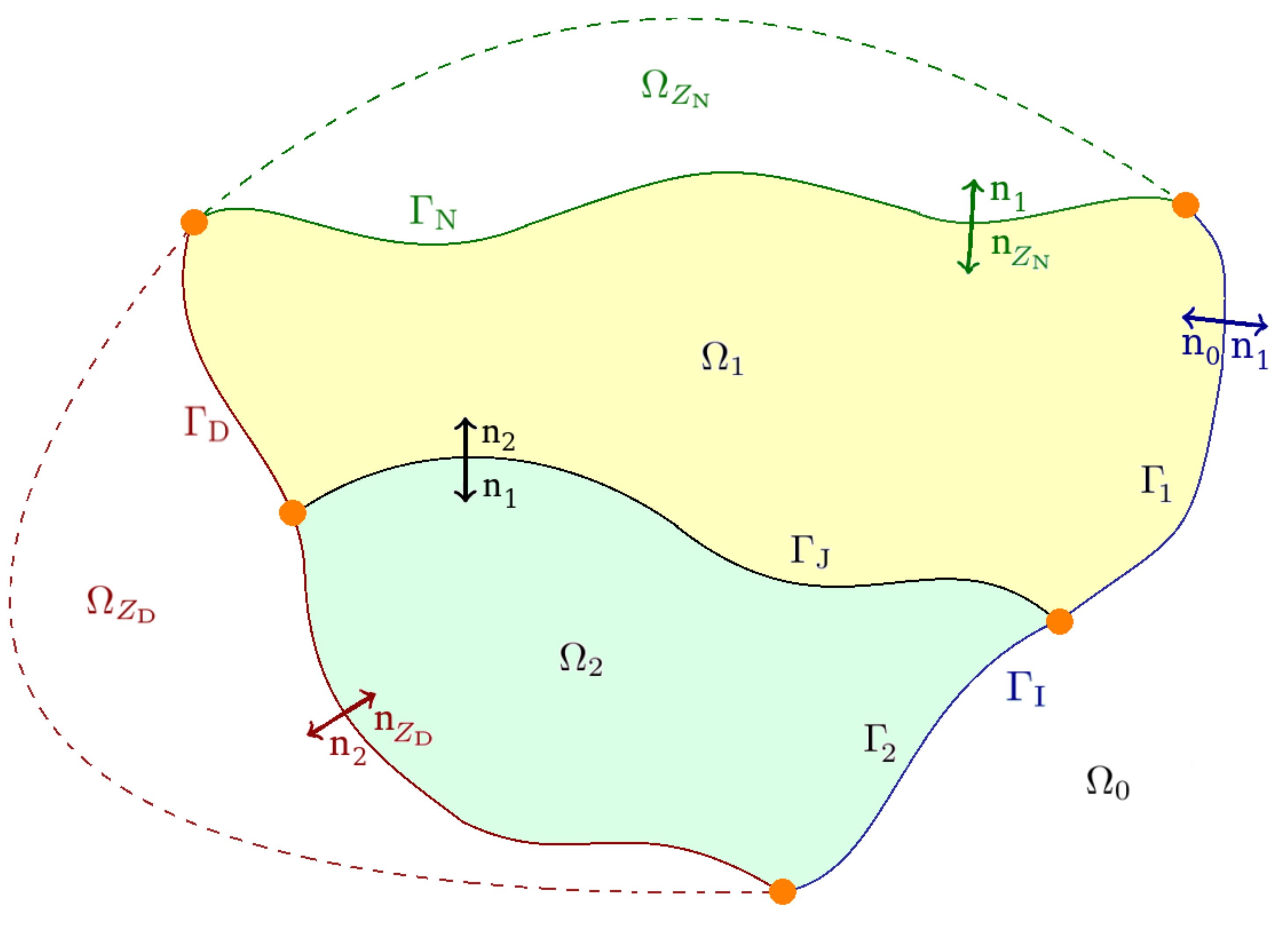}
  \end{center}  
  \caption{Cross section of the computational domain.
    The domain $\Omega$ is split into the disjoint open sets $\Omega_1, \Omega_2$, corresponding to different materials, and their interface $ \Gamma_{\dataJump} \coloneqq \overline \Omega_1 \cap \overline \Omega_2$.
    We set $\Gamma_1 \coloneqq \partial \Omega_1$ and $\Gamma_2 \coloneqq \partial \Omega_2$.
    The unbounded exterior domain is denoted by $\Omega_0 \coloneqq \mathbb{R}^3 \setminus \overline{\Omega}$.
    For $j = 0,1,2$ the unit normal vector pointing outside $\Omega_j$ is denoted by $\normalN_j$.
    The auxiliary domains $\Omega_{Z_{\dataDirichlet}},\Omega_{Z_{\dataNeumann}}$ are used in the definition of $\XSingleZero$ (cf. (\ref{DefXsinglespaceD}{})).}
  \label{domain}
\end{figure}

We consider a bounded Lipschitz domain $\Omega \subseteq \mathbb R^3$ partitioned as in Fig.~\ref{domain} by an interface $\Gamma_{\dataJump}$ into $\Omega_1, \Omega_2$.
The results of this
paper also hold for exterior problems in $\Omega_{0}:=\mathbb{R}%
\backslash\overline{\Omega}$ with bounded interface $\Gamma_{\dataJump}$ lying
in the exterior domain, and when splitting $\Omega$ to any finite number of
subdomains. In order to simplify the notations, we restrict ourselves to the
case of the interior problem and two subdomains.

For $i\in\{0,1,2\}$ we set $\Gamma_{i}\coloneqq\partial\Omega_{i}$ and employ
the convention that $n_{i}$ is the unit normal vector field at $\partial
\Omega_{i}$ pointing into the exterior of $\Omega_{i}$.
The \emph{skeleton manifold} is defined by $\Sigma \coloneqq \Gamma_1 \cup \Gamma_2$.

We also introduce a partition of the boundary, corresponding to different types of boundary conditions (see again Fig.~\ref{domain}):
we split $\Sigma = \Gamma_{\dataDirichlet} \cup \Gamma_{\dataNeumann} \cup \Gamma_{\dataImpedance} \cup \Gamma_{\dataJump}$;
transmission (jump) conditions will be imposed at $\Gamma_{\dataJump}$, Dirichlet boundary conditions at $\Gamma_{\dataDirichlet}$, Neumann boundary conditions at $\Gamma_{\dataNeumann}$, and an impedance condition at $\Gamma_{\dataImpedance}$;
we do not require $\Gamma_{\dataDirichlet}, \Gamma_{\dataNeumann}, \Gamma_{\dataImpedance}$ to be connected -- however we assume the relative interiors of these subsets are disjoint.

The new mathematical aspect of our setting is the presence of an interface
\emph{and} general mixed boundary conditions of Dirichlet, Neumann, and
impedance type. We do not impose restrictions on where the interface meets the
domain boundary.

The resulting transmission initial-boundary value problem to be solved for
$u\in H^{1}\left(  \left[  0,T\right]  \times\Omega\right)  $ is
\begin{subequations}
  \label{trans_probl}
  \begin{align}[left= {\empheqlbrace}]
    p_1^2 \partial_t^2 u_1 - a_1^2 \Delta u_1 =                                                        & 0 \text{ in } \Omega_1 \times [0,T],                                                       \\
    p_2^2 \partial_t^2 u_2 - a_2^2 \Delta u_2 =                                                        & 0 \text{ in } \Omega_2 \times [0,T],                                                       \\
    [u]_{\Gamma_{\dataJump}} = \left[ a^2 \frac{\partial u}{\partial n} \right]_{\Gamma_{\dataJump}} = & 0 \text{ on } \Gamma_{\dataJump} \times [0,T], \label{trans_probl:J}                       \\
    u =                                                                                                & g_{\dataDirichlet} \text{ on } \Gamma_{\dataDirichlet} \times [0,T], \label{trans_probl:D} \\
    a^2 \frac{\partial u}{\partial n} - \timeOperator T * \dot u =                                     & d_{\dataImpedance} \text{ on } \Gamma_{\dataImpedance} \times [0,T], \label{trans_probl:I} \\
    a^2 \frac{\partial u}{\partial n} =                                                                & d_{\dataNeumann} \text{ on } \Gamma_{\dataNeumann} \times [0,T], \label{trans_probl:N}     \\
    u(0, \cdot) = \dot u(0,\cdot) =                                                                    & 0 \text{ in } \Omega;
  \end{align}
\end{subequations}
here and in the following, we employ the shorthand $\dot{u}$ for $\partial_{t} u$;
$*$ denotes the convolution in time, for $i \in \{1,2\}$, $u_i \coloneqq \restrict{u}_{\Omega_i}$ and the functions $a$ and $p$ are defined on $\Omega_0 \cup \Omega_1 \cup \Omega_2$:
\begin{equation}
  \label{defatau}
  \left. a\right\vert _{\Omega_{\ell}}=a_{\ell},\quad\left. p\right\vert_{\Omega_{\ell}}=p_{\ell},\quad\ell=0,1,2
\end{equation}
via the material-dependent constant coefficients $a_1, a_2, p_1, p_2 > 0$.
They are extended to positive functions $a_0(x)$, $p_0(x)$ to the exterior domain $\Omega_0$, such that $a,p$ are continuous \emph{across} the interface $\Gamma_0$, while they are, in general, discontinuous \emph{along} $\Gamma_{0}$ at points where the interface meets
$\partial\Omega$.
The temporal convolution operator $\timeOperator T$ may depend on $p$ and $a$.
For the boundary data we assume (postponing the introduction of the relevant Sobolev spaces to Section~\ref{SecCaldOp}):
\begin{equation*}
  g_{\dataDirichlet} \in \tilde{H}_{\dataImpedance}^{1/2}(\Gamma_{\dataDirichlet}),\qquad
  d_{\dataNeumann} \in \tilde H^{-1/2}_{\dataImpedance}(\Gamma_{\dataNeumann}),\qquad
  d_{\dataImpedance}\in \tilde H^{-1/2}_{\dataNeumann}(\Gamma_{\dataImpedance}).
\end{equation*}
In (\ref{trans_probl:J}), the direction is not relevant;
$\left[
\cdot\right]  _{\Gamma_{\dataJump}}$ denotes the jump of a function across the
interface $\Gamma_{\dataJump}$. The temporal convolution operator
$\mathcal{T}$ is a Dirichlet-to-Neumann ($\operatorname*{DtN}$) operator or an
approximation to it.
The simplest approximation is given by impedance boundary conditions: $\timeOperator{T}(t) = -a p \delta_{0}(t)$, where $\delta_{0}$ is the Dirac distribution.
At this point we are vague concerning the function spaces which are mapped by $\timeOperator{T}(t)$ in a continuous way but postpone this to Section~\ref{sec:TraceOperatorsAndTraceSpaces}, where also a dissipative condition will be imposed on $\timeOperator{T}(t)$ (Assumption \ref{A1}).

\subsection{Retarded Potential Integral Equations}
\label{sec:RetardedPotentialIntegralEquations}
To formulate a RPIE we need trace operators.
For vector-valued functions $\vectorValued w$, sufficiently smooth in $\overline{\Omega_i}$, we define the normal component trace by
\begin{equation}
  \label{conormalcomp}
  \traceNormal{i} \vectorValued{w} \coloneqq \left\langle \normalN_i, \restrict{\vectorValued{w}}_{\Gamma_i} \right\rangle,\qquad i \in \{0,1,2\}
\end{equation}
where for $\vectorValued v = \transpose{(v_1, v_2, v_3)}, \vectorValued w = \transpose{(w_1, w_2, w_3)} \in \mathbb{C}^3$ we set $\left\langle \boldsymbol{v},\boldsymbol{w}\right\rangle \coloneqq\sum_{j=1}^{3}v_{j}w_{j}$ (without complex conjugation) and the unit normal vector $n_i$ points outside $\Omega_i$.

For $u$ sufficiently regular in $\overline{\Omega_i}$ and $a$ as in (\ref{defatau}), the Dirichlet (D) and Neumann (N) trace operators are denoted by $\traceD{i}, \traceN{i}$ and are given by
\begin{equation}
  \label{deftraces}
  \begin{aligned}
    \traceD{i} u \coloneqq & \left( \restrict{u}_{\Omega_i} \right)\vert_{\Gamma_i}, & \traceN{i} u \coloneqq & \traceNormal{i} \left( a_i^2 \gradient \restrict{u}_{\Omega_i} \right),
  \end{aligned}
\end{equation}
where the index $i \in \{0,1,2\}$ indicates that the limit is taken from the subdomain $\Omega_i$.
We also need a notation for the case where the limit of a function $u$ regular enough in the complement $\Omega_i^c \coloneqq \mathbb{R}^3 \setminus \Omega_i$ is taken from outside $\Omega_i$ (and the unit normal $\normalN_i$ still points outside $\Omega_{i}$):
\begin{equation*}
  \begin{aligned}
    \complementTrace{\traceD{i}} u \coloneqq & \left( u|_{\Omega_i^c} \right)\vert_{\Gamma_i}, & \complementTrace{\traceN{i}} u \coloneqq & \traceNormal{i} \left( a_i^2 \gradient \restrict{u}_{\Omega_i^c} \right).
  \end{aligned}
\end{equation*}

Each part of the skeleton $\Sigma$ is endowed with an intrinsic orientation.
We introduce (for $j=1,2$) the orientation functions $\mathfrak{N}_j : \Gamma_j \to \{-1,1\}$ to take into account its compatibility with the induced orientations on $\Gamma_j$:
\begin{equation}
  \label{eq:orientationSigma}
  \mathfrak{N}_j(x) = \left\langle \normalN_j(x), \normalN_{\Sigma}(x) \right\rangle \text{ for all } x \in \Gamma_j.
\end{equation}
However we assume that $\normalN_{\Sigma}$ always points outside $\Omega$ on $\partial \Omega$.

At this point we can define, for $u$ regular enough in $\Omega$,
\begin{equation*}
  \traceDFull u \coloneqq \restrict{\restrict{u}_{\Omega}}_{\Sigma}, \qquad \traceNFull u \coloneqq \left\langle \normalN_{\Sigma}, \restrict{a^2 \gradient \restrict{u}_{\Omega}}_{\Sigma} \right\rangle.
\end{equation*}
Finally, we will use the same symbols for the continuous extensions of the trace operators to appropriate Sobolev spaces.

We also need the potential $\timeOperator{G}_i$: for $\vectorValued \varphi = \transpose{(\varphi_{\dataDirichlet},\varphi_{\dataNeumann})} \in \XTraces{i}$
\begin{equation*}
  ( \timeOperator{G}_i * \vectorValued \varphi )(t,x) \coloneqq \int_{0}^{t} \int_{\Omega_i} k_i (t-\tau, x-y) \varphi_{\dataNeumann} - \gamma_{\dataNeumann,i;y} k_i (t-\tau, x-y) \varphi_{\dataDirichlet} \de y \de \tau,
\end{equation*}
where $\gamma_{\dataNeumann,i;y}$ denotes the co-normal derivative with respect to the $y$-variable;
for $i\in\{1,2\}$ the kernel function $k_{i}$ and the (Cauchy-trace) space
$\boldsymbol{X}_{i}$ will be defined in (\ref{eq:fundamentalSolution}) and
(\ref{eq:multiSpaces}).

Kirchhoff's representation formula then gives for the solution $u$ of (\ref{trans_probl}) (recall that for $i \in \{1,2\}$, $u_i \coloneqq \restrict{u}_{\Omega_i}$)
\[
u_{i}=\mathcal{G}_{i}\ast\boldsymbol{\gamma}_{i}u_{i},
\]
and applying the trace operator on both sides leads to the Calderón identity%
\[
\boldsymbol{\gamma}_{i}u_{i}=\boldsymbol{\gamma}_{i}\mathcal{G}_{i}%
\ast\boldsymbol{\gamma}_{i}u_{i}.
\]
By inserting the initial and boundary data (\ref{trans_probl:D},\ref{trans_probl:N}) and the equation (\ref{trans_probl:I}) one ends up with a system of \emph{retarded potential boundary integral equations} for the unknown Cauchy data of the boundary $\partial\Omega$ and interface $\Gamma_{\dataJump}$:
\begin{equation}
  \label{eq:rbieInitial}
  \left(
    \begin{bmatrix}
      \begin{matrix}
        -\mathcal{K}_{1} & \mathcal{V}_{1}\\
        \mathcal{W}_{1} & \mathcal{K}_{1}^{\prime}
      \end{matrix}
      & \\
      &
      \begin{matrix}
        -\mathcal{K}_{2} & \mathcal{V}_{2}\\
        \mathcal{W}_{2} & \mathcal{K}_{2}^{\prime}
      \end{matrix}
    \end{bmatrix}
    -\frac{\boldsymbol{\delta}_{0}}{2}\right) \ast
  \begin{pmatrix}
    \gamma_{\dataDirichlet,1}u\\
    \gamma_{\dataNeumann,1}u\\
    \gamma_{\dataDirichlet,2}u\\
    \gamma_{\dataNeumann,2}u
  \end{pmatrix} = \vectorValued{0},
\end{equation}
where the known boundary data are given by (\ref{trans_probl:D} -- \ref{trans_probl:N}), and incorporated as part of the unknown traces.
We emphasize that there are several ways to include boundary and jump conditions and we explain our approach via ``single trace spaces'' in Section~\ref{sec:TraceOperatorsAndTraceSpaces}.
Here, $\mathcal{V}_{i}$, $\mathcal{K}_{i}$, $\mathcal{K}_{i}^{\prime}$, $\mathcal{W}_{i}$ are scalar retarded potential integral operators (RPIOs) (defined in Section \ref{SecLPBIE}).

On the interface we have two sets of traces: those from $\Omega_{1}$ and those
from $\Omega_{2}$ and, in order to close this system of RPIEs, we supplement
it by the interface conditions (\ref{trans_probl:J})%
\begin{equation}
\lbrack u]_{\Gamma_{\dataJump}}=\left[  a^{2}\frac{\partial u}{\partial
n}\right]  _{\Gamma_{\dataJump}}=0\text{ on }\Gamma_{\dataJump}\times
\lbrack0,T].\label{trans_probl:Jrep}%
\end{equation}
In our approach, we will eliminate these coupling conditions by employing a
\textit{single-trace} ansatz (cf. \cite{ClHiptJH}) which automatically ensures
(\ref{trans_probl:Jrep}).

\subsection{Outline and main results}
\label{sec:OutlineAndMainResults}
For the analysis of the above system of RPIE (as well as for applying the
convolution quadrature for its numerical solution), these equations are
transformed to a system of \emph{integro-differential equations} in the
frequency domain. For this, equation (\ref{eq:rbieInitial}) is considered as a
convolution equation of the abstract form%
\begin{equation}
(\boldsymbol{\mathcal{O}}\ast\boldsymbol{\phi})(t)=\boldsymbol{r(t}%
),\quad\forall t\in\left[  0,T\right]  .\label{abstractconv}%
\end{equation}
The unknown function $\boldsymbol{\phi}:\left[  0,T\right]  \rightarrow
\boldsymbol{X}$ maps to an appropriate function space $\boldsymbol{X}$. If the
operator $\boldsymbol{\mathcal{O}}$ is replaced by the inverse Laplace
transform of its Laplace transform $\vectorValued{\laplaceOperator O}$:
\begin{equation*}
  \vectorValued{\timeFunction r}(t) = \int_0^t \frac{1}{2 \pi \imaginary} \int_{c + \imaginary \mathbb R} e^{s(t-\tau)} \vectorValued{\laplaceOperator O}(s) \de s \vectorValued{\timeFunction \phi}(\tau) \de \tau
  = \frac{1}{2 \pi \imaginary} \int_{c + \imaginary \mathbb R} \vectorValued{\laplaceOperator O}(s) \int_0^t e^{s(t-\tau)} \vectorValued{\timeFunction \phi}(\tau) \de \tau \de s,
\end{equation*}
the inner integral $\boldsymbol{z}(s;t)\coloneqq\int_{0}^{t}e^{s(t-\tau
)}\boldsymbol{\phi}(\tau)\mathop{}\!\mathrm{d}\tau$ is the solution of the
initial value problem
\[
\boldsymbol{\dot{y}}(t)=s\boldsymbol{y}(t)+\boldsymbol{\phi}\qquad
\boldsymbol{y}(0)=\boldsymbol{0}\quad\forall t\in\left[  0,T\right]  .
\]
The convolution equation can be reformulated as the following system for the unknown $\vectorValued{\phi}$ and the auxiliary function $\vectorValued{z}$, for some $\sigma_{0}>0$:
\begin{equation}
\left.
\begin{array}
[c]{c}%
\displaystyle\frac{1}{2\pi\mathrm{i}}\displaystyle\int_{\sigma_{0}+\operatorname*{i}\mathbb{R}%
}\boldsymbol{\mathsf{O}}\left(  s\right)  \boldsymbol{z}%
(s,t)\mathop{}\!\mathrm{d}s=\mathbf{r}\left(  t\right) \\
\partial_{t}\boldsymbol{z}(s;t)=s\boldsymbol{z}(s;t)+\boldsymbol{\phi
}(t),\ \boldsymbol{z}(s;0)=\boldsymbol{0}%
\end{array}
\right\}  \quad\forall t\in\left[  0,T\right]  ,\quad\forall s\in\sigma
_{0}+\operatorname*{i}\mathbb{R}\label{eq:integroDifferential}.
\end{equation}
The solution $\boldsymbol{\phi}$ of
(\ref{eq:integroDifferential}) is then also the solution of
(\ref{abstractconv}).

\begin{remark}
  \label{RemLaplTD}
  The analysis of the Laplace-transformed retarded potential integral operator (RPIO) is key for the analysis of the system of RPIE (\ref{eq:rbieInitial}) since well-posedness results can be transferred from the Laplace to the time domain via the Herglotz theorem, see \cite{banjai_coupling}.
  For the numerical discretization of (\ref{eq:integroDifferential}) by convolution quadrature, the starting point is the discretization of the ODE in (\ref{eq:integroDifferential}) by a time stepping method.
  Also here, the error analysis relies on frequency-explicit coercivity and continuity properties of the integral operator in the Laplace domain.
\end{remark}

\paragraph{Main results}
In this paper, we will derive a formulation of the wave transmission problem with mixed boundary conditions as a retarded potential integral equation for a single trace space of the form (\ref{abstractconv}) as well as an equivalent integro-differential equations of the form (\ref{eq:integroDifferential}).

Our main theoretical result is the proof of well-posedness of the RPIE (\ref{eq:rbieInitial} -- \ref{trans_probl:Jrep}).
This will be obtained by the methodology as explained in Remark \ref{RemLaplTD} by proving coercivity and continuity of the Laplace-transformed RPIO.

\paragraph{Organization of the paper}
Sections~\ref{SecLPBIE}--\ref{sec:dirichlet} are devoted to the derivation of the system of RPIEs;
the retarded acoustic single and double layer potentials are defined and the corresponding boundary integral operators are introduced by applying the trace and normal trace operator to these potentials.
We end up with a system of integral equations for the unknown Cauchy data.
Note that we employ a single-trace ansatz which involves single Cauchy data across the interface in accordance with the transmission conditions.

In Section~\ref{SecLinComb} we propose to incorporate the impedance boundary condition by keeping both Cauchy data in the equation.
The advantage of this approach is that only boundary integral operators are involved which are defined on closed surfaces.

In Section~\ref{sec:WellPosednessAnalysis} we will prove well-posedness of the system of integral equations by showing coercivity and continuity of this system of RPIEs.
This allows us to determine the analyticity class of the Laplace-transformed system and implies existence and uniqueness.

\section{Retarded Boundary Integral Equations for the Wave Transmission Problem\label{SpaceTimeBEM}}
After having sketched the approach we will detail here the operators, function spaces and Calderón identities, and formulate the wave transmission problem with mixed boundary conditions (\ref{trans_probl}) as a retarded boundary integral equation in variational form (see (\ref{eq:variationalTime:Tranmission_direct_form_ext})) for the unknown boundary traces.
This requires some preliminaries:
first, we introduce the relevant boundary integral operators (Section~\ref{SecLPBIE}).
We have chosen the \emph{direct} approach based on Kirchhoff's representation formula (Section~\ref{sec:representationFormula}, (\ref{CaldProjEq})) which involves the Calderón projector.
This operator is expressed in the Laplace domain by using the block operator $\vectorValued{\laplaceOperator{A}}(s)$ (see (\ref{defAsj})) which is also needed for the definition of the sesquilinear form in the variational formulation (\ref{eq:a0l0}).
In Section~\ref{sec:dirichlet} we incorporate the Dirichlet and Neumann boundary conditions and finally, in Section~\ref{SecLinComb}, we take into account the impedance-type condition.
The variational formulation of the RPIE in the Laplace domain is formulated as Problem~\ref{problem:mixedFormulation}  while the equation in the time domain is presented in (\ref{eq:variationalTime:Tranmission_direct_form_ext}).

\subsection{Background: Layer Potentials and Boundary Integral Operators\label{SecLPBIE}}

We recall retarded potentials on two-dimensional compact, orientable manifolds in $\mathbb{R}^{3}$ and start by introducing some notation.
We write $\Gamma_{i,S}\coloneqq\Gamma_{i}\cap\Gamma_{S}$ for $S \in \{\dataDirichlet,\dataNeumann,\dataImpedance,\dataJump\}$ i.e., the index $i \in \left\{ 0,1,2 \right\}$ corresponds to the domain $\Omega_{i}$ while $S$ indicates the type of boundary conditions imposed.

Recall the definition of $a$ as in (\ref{defatau}).
Let $u$ be a function in $\mathbb{R}^{3} \setminus \Sigma$;
for $j\in\{0,1,2\}$, we assume that the traces $\traceD{j},\traceN{j},\complementTrace{\traceD{j}},\complementTrace{\traceN{j}}$ applied to $u$ are well-defined.
Then the jump $\left[ \cdot \right]_{\dataDirichlet,j}$ and co-normal jump $\left[ \cdot \right]_{\dataNeumann,j}$ across $\Gamma_j$ are defined by
\begin{equation*}
  [u]_{\dataDirichlet,j} \coloneqq \complementTrace{\traceD{j}} u - \traceD{j} u \qquad \text{and} \qquad [u]_{\dataNeumann,j} \coloneqq \complementTrace{\traceN{j}} u - \traceN{j} u.
\end{equation*}
The \textit{averages} are defined by
\begin{equation*}
  \{u\}_{\dataDirichlet,j} \coloneqq \frac{1}{2} \left( \traceD{j} u + \complementTrace{\traceD{j}} u \right) \qquad \text{and} \qquad \{u\}_{\dataNeumann,j} \coloneqq \frac{1}{2} \left( \traceN{j} u + \complementTrace{\traceN{j}} u \right).
\end{equation*}
This allows us to introduce the following boundary integral operators.
The fundamental solution of the wave equation in $\mathbb{R}^{3}$, more precisely, for the operator $p_i^2 \partial_t^2 - a_i^2 \Delta$ is (see e.g., in the Laplace domain: \cite[p.~486, (18)]{Stratton_book}, \cite[Eq.~(2.10)]{sayas_book}; in cylindrical coordinates: \cite{friedman}):
\begin{equation}
  \label{eq:fundamentalSolution}
  k_{i}(t,z) \coloneqq \frac{\delta_0 \left( t - \frac{p_i}{a_i} \norm{z} \right)}{4\pi a_i^2 \norm{z}} \text{ for } z \in \mathbb R^3 \setminus \{0\}.
\end{equation}

Let the coefficient functions $a$, $p$ be as in (\ref{defatau}).
For functions $\varphi: [0,T] \times \Gamma_{i} \to\mathbb{C}$ and $\psi: [0,T] \times \Gamma_{i} \to \mathbb{C}$ we define the retarded acoustic single and double layer potentials for all $(t,x) \in [0,T] \times \left( \mathbb{R}^{3} \setminus\Gamma_{i} \right)$:
\begin{subequations}
  \label{defpots}
  \begin{align}
    (\timeOperator{S}_i \ast \varphi )(t,x) & \coloneqq \int_{\Gamma_i} (k_i (\cdot,\norm{x-y}) * \varphi(\cdot,y))(t) \de s_{y}=\int_{\Gamma_i} \frac{\varphi \left(t-\frac{p_i \norm{x-y}}{a_i}, y \right)}{4 \pi a_i^2 \norm{x-y}} \de s_y,\label{defpotsa} \\
    (\timeOperator{D}_i \ast \psi ) (t,x)   & \coloneqq\int_{\Gamma_i} \restrict{\gamma_{\dataNeumann,i;y} (k_i(\cdot, \norm{x-y}) * \psi(\cdot, z))(t)}_{z=y} \de s_{y}\nonumber                                                                               \\
                                            & =\int_{\Gamma_i} \restrict{\left(\gamma_{\dataNeumann;i;y}\frac{\psi\left( t - \frac{p_{i} \norm{x-y}}{a_{i}}, z \right)}{4\pi a_{i}^{2} \norm{x-y}}\right)}_{z=y} \de s_{y}.\label{defpotsb}
  \end{align}
\end{subequations}
In \cite[Eq.~(10)]{duong03} an explicit expression for the integrand of the double layer potential is provided.

These potentials give rise to the following boundary integral operators.
For functions $\varphi,\psi: [0,T] \times \Gamma_{j} \to \mathbb{C}$ we set
\begin{align*}
\mathcal{V} _{j} \ast\varphi &  \coloneqq\left\{  \mathcal{S} _{j} \ast
\varphi\right\} _{\dataDirichlet;j}, & \mathcal{K} _{j}\ast\psi &
\coloneqq \left\{  \mathcal{D} _{j} \ast\psi\right\} _{\dataDirichlet;j},\\
\mathcal{K} ^{\prime}_{j} \ast\varphi &  \coloneqq \left\{  \mathcal{S} _{j}
\ast\varphi\right\} _{\dataNeumann;j}, & \mathcal{W} _{j}\ast\psi &
\coloneqq-\left\{  \mathcal{D} _{j}\ast\psi\right\} _{\dataNeumann;j}%
\end{align*}
on $[0,T] \times \Gamma_{j}$.
For $j \in\{0,1,2\}$, it holds almost everywhere on $[0,T] \times \Gamma_{j}$
\begin{align*}
\gamma_{\dataDirichlet, j} (\mathcal{S} _{j} * \varphi)  &  = \mathcal{V} _{j}
* \varphi, & \gamma_{\dataNeumann, j} (\mathcal{S} _{j} * \varphi)  &  =
\left(  \mathcal{K} _{j}^{\prime} + \frac{\delta_{0}}{2} \right)  * \varphi,\\
\gamma_{\dataDirichlet, j} (\mathcal{D} _{j} * \psi)  &  = \left(  \mathcal{K}
_{j} - \frac{\delta_{0}}{2} \right)  * \psi, & \gamma_{\dataNeumann, j}
(\mathcal{D} _{j} * \psi)  &  = - \mathcal{W} _{j} * \psi.
\end{align*}

For $\kappa\in\mathbb{R}$, let%
\[
\mathbb{C}_{\kappa}\coloneqq\left\{  s\in\mathbb{C}\mid\real s>\kappa\right\}
.
\]

\begin{convention}
Throughout this paper, $\sigma_{0}>0$ denotes a fixed positive constant. The
constants in the estimates in this paper will depend continuously on
$\sigma_{0}\in\mathbb{R}_{>0}$ and $a_{1},a_{2},p_{1},p_{2}\in\mathbb{R}_{>0}$
in (\ref{defatau}). These constants, possibly, tend to infinity if one or more
of the quantities $\sigma_{0}$, $a_{1}$, $a_{2}$, $p_{1}$, $p_{2}$ tend to
zero or infinity. We will suppress this dependence in our notation.
\end{convention}

We employ the convention that, if the two functions $\timeFunction \varphi$ and $\laplaceFunction{\varphi}$ appear in the same context, then the latter is the Laplace transform of the former.
We recall the formal definition of the Laplace transform $\mathcal{L}$ and its inverse $\mathcal{L}^{-1}$ by
\begin{equation*}
  \laplaceFunction q (s) \coloneqq \left( \mathcal{L} q \right)(s) = \int_{0}^{\infty} \operatorname{e}^{-st} q(t) \de t \qquad \text{and} \qquad
  q(t) = \left( \mathcal{L}^{-1} \laplaceFunction{q} \right) (t) = \frac{1}{2 \pi \imaginary} \int_{\sigma_0 + \imaginary \mathbb{R}} \operatorname{e}^{st} \laplaceFunction{q}(s) \de s.
\end{equation*}
For the convolution quadrature, we apply the Laplace transform with respect to
time and obtain operators in the frequency variable $s \in\mathbb{C}_{0}$.
Thus, we end up with the Laplace transformed potentials for $(s,x) \in \mathbb{C}_{0} \times \mathbb{R}^{3} \setminus\overline{\Gamma_{i}}$ and $i\in\{0,1,2\}$:
\begin{subequations}
\label{freq_pot}%
\begin{align}
\left(  \mathsf{S} _{i} (s) \varphi\right)  (x)  &  \coloneqq \int_{\Gamma
_{i}} \hat{k}_{i} (s, x-y) \varphi(y) \mathop{}\!\mathrm{d} s_{y}
\label{freq_pota},\\
\left(  \mathsf{D} _{i} (s) \psi\right)  (x)  &  \coloneqq \int_{\Gamma_{i}}
\left(  \gamma_{\dataNeumann;i;y} \hat{k}_{i} (s,x-y) \right)  \psi(y)
\mathop{}\!\mathrm{d} s_{y},\label{freq_potb}%
\end{align}
for
\end{subequations}
\begin{equation*}
  \laplaceFunction k_i (s,z) \coloneqq \frac{\exp \left( -s \frac{p_i \norm{z}}{a_i} \right)}{4 \pi a_i^2 \norm{z}},\ z \in \mathbb R^3 \setminus \{0\}
\end{equation*}
and corresponding boundary integral operators on $\Gamma_{j}$ given for $s
\in\mathbb{C}_{0}$ by
\begin{align*}
\mathsf{V} _{j} (s) \varphi &  \coloneqq \left\{  \mathsf{S} _{j} (s)
\varphi\right\} _{\dataDirichlet;j}, & \mathsf{K} _{j} (s) \psi &
\coloneqq \left\{ \mathsf{D} _{j} (s) \psi\right\} _{\dataDirichlet;j},\\
\mathsf{K} _{j}^{\prime}(s) \varphi &  \coloneqq \left\{  \mathsf{S} _{j} (s)
\varphi\right\} _{\dataNeumann;j}, & \mathsf{W} _{j} (s) \psi &
\coloneqq -\left\{ \mathsf{D} _{j} (s) \psi\right\} _{\dataNeumann;j}.
\end{align*}

Note that the Laplace transform $\mathcal{L}$ applied to the convolution potentials satisfies
\[
\mathcal{L} \left(  \mathcal{S} _{i} \ast\varphi\right)  (s) = \mathsf{S} _{i}
(s) \hat{\varphi} (s),\qquad\mathcal{L} \left(  \mathcal{D} _{i} \ast
\psi\right)  (s) = \mathsf{D} _{i} (s) \hat{\psi} (s)
\]
and analogous relations hold for the boundary integral operators in the time
and Laplace domain. It is also well known that the following jump relations
hold (see \cite[Section~1.3]{sayas_book}):
\begin{equation}
\label{jump_rel_Lapl}\begin{aligned} \left[ \ensuremath{\mathsf{S}}_j (s) \varphi \right]_{\dataDirichlet;j} & = 0, & \left[ \ensuremath{\mathsf{S}}_j (s) \varphi \right]_{\dataNeumann;j} & = -\varphi, \\ \left[ \ensuremath{\mathsf{D}}_j (s) \psi \right]_{\dataDirichlet;j} & = \psi, & \left[ \ensuremath{\mathsf{D}}_j (s) \psi \right]_{\dataNeumann;j} & = 0. \end{aligned}
\end{equation}

\subsection{Representation Formula}
\label{sec:representationFormula}
\subsubsection{Sobolev Spaces}
First, we introduce Sobolev spaces in domains and on manifolds -- standard
references are \cite{Adams_new}, \cite{LionsMagenesI}. Let $\Omega
\subset\mathbb{R}^{3}$ be a bounded Lipschitz domain with boundary $\Gamma$.
The unit normal vector field $n$ on $\Gamma$ is chosen to point into the
exterior of $\Omega$ and exists almost everywhere. We denote the $L^{2}%
(\Omega)$-scalar product and norm by
\begin{equation*}
  (u,v)_{\Omega} \coloneqq \int_{\Omega} u(x) \overline{v}(x) \de x \qquad \text{and} \qquad \norm{u}_{\Omega}\coloneqq(u,u)_{\Omega}^{1/2},
\end{equation*}
and suppress the subscript $\Omega$ if the domain is clear from the context.
For $\alpha\in\mathbb{R}_{\geq0}$, let $H^{\alpha}\left(  \Omega\right)  $
denote the usual Sobolev space with norm $\norm{\cdot}_{H^{\alpha}\left(
\Omega\right)  }$ and $H_{0}^{\alpha}\left(  \Omega\right)  $ is the closure
of $C_{0}^{\infty}\left(  \Omega\right)  \coloneqq\left\{  u\in C^{\infty
}\left(  \Omega\right)  \mid\operatorname*{supp}u\subset\Omega\right\}  $ with
respect to the $\norm{\cdot}_{H^{\alpha}\left(  \Omega\right)  }$ norm. Its
dual space is denoted by $H^{-\alpha}\left(  \Omega\right)  \coloneqq\left(
  H_{0}^{\alpha}\left(  \Omega\right)  \right)  ^{\prime}$.
On the boundary $\Gamma$, we define the Sobolev space $H^{\alpha}(\Gamma)$, $\alpha\geq0$, in the usual way.
Note that the range of $\alpha$ for which $H^{\alpha}(\Gamma)$ is defined may be limited, depending on the global smoothness of the surface $\Gamma$;
for Lipschitz surfaces, $\alpha$ can be chosen in the range $\left[  0,1\right]  $;
for $\alpha<0$, the space $H^{\alpha}(\Gamma)$ is the dual of $H^{-\alpha}\left(  \Gamma\right)$ (see, e.g., \cite[p.~98]{Mclean00}).

We define, for $R,S\in\{\dataDirichlet,\dataNeumann,\dataImpedance\},R\neq S$, the Sobolev spaces
\begin{align}
  H^{\pm 1/2}(\Gamma_R) \coloneqq             & \left\{ \restrict{\phi}_{\Gamma_R} \text{ such that } \phi \in H^{\pm 1/2} (\Sigma) \right\},\label{eq:noTilde} \\
  \tilde{H}_{S}^{\pm 1/2}(\Gamma_R) \coloneqq & \left\{ \restrict{\phi}_{\Gamma_R} \text{ such that } \phi \in H^{\pm 1/2} (\Sigma) \text{ and } \restrict{\phi}_{\Gamma_S} = 0 \right\}.\label{eq:halfTilde}
\end{align}

We denote by $\left\langle \cdot,\cdot\right\rangle _{\Gamma_{j}}$ the dual
pairing between $H^{1/2}\left(  \Gamma_{j}\right)  $ and $H^{-1/2}\left(
\Gamma_{j}\right)  $ (without complex conjugation) so that $\left\langle
u,\overline{v}\right\rangle _{\Gamma_{j}}$ is the continuous extension of the
$L^{2}\left(  \Gamma_{j}\right)  $ scalar product. We can thus introduce the
symmetric and skew-symmetric dual pairing:
for $j=1,2$ and $\vectorValued{\phi} = \transpose{(\phi_{\dataDirichlet},\phi_{\dataNeumann})}, \vectorValued{\psi} = \transpose{(\psi_{\dataDirichlet},\psi_{\dataNeumann})} \in H^{1/2}(\Gamma_j) \times H^{-1/2}(\Gamma_j)$
\begin{subequations}
  \label{eq:pairings}
  \begin{align}
    \left\langle \vectorValued{\phi}, \vectorValued{\psi} \right\rangle_{\Gamma_j}^+ & \coloneqq \left\langle \phi_{\dataDirichlet},\psi_{\dataNeumann} \right\rangle_{\Gamma_j} + \left\langle \phi_{\dataNeumann}, \psi_{\dataDirichlet}\right\rangle_{\Gamma_j}, \label{eq:sympairing}\\
    \left\langle \vectorValued{\phi}, \vectorValued{\psi} \right\rangle_{\Gamma_j}^- & \coloneqq \left\langle \phi_{\dataDirichlet},\psi_{\dataNeumann} \right\rangle_{\Gamma_j} - \left\langle \phi_{\dataNeumann}, \psi_{\dataDirichlet}\right\rangle_{\Gamma_j}. \label{eq:skewsympairing}
  \end{align}
\end{subequations}

\subsubsection{Trace Operators and Trace Spaces}
\label{sec:TraceOperatorsAndTraceSpaces}

Note that the trace operators $\gamma_{\dataDirichlet;i},\gamma
_{\dataNeumann;i}$ in (\ref{deftraces}) can be extended to continuous
operators acting on functions in the Sobolev space $H\left(  \Delta,\Omega
_{i}\right)  \coloneqq\left\{  u\in H^{1}\left(  \Omega_{i}\right)  \mid\Delta
u\in L^{2}\left(  \Omega_{i}\right)  \right\}  $. We collect the range of
these traces into the space of Cauchy traces, and the \emph{multi-trace} space:
\begin{equation}
  \label{eq:multiSpaces}
  \begin{aligned}
    \XTraces{i} \coloneqq H^{1/2} \left( \Gamma_i \right) \times H^{-1/2} \left( \Gamma_i \right) \text{ for } i \in \{1,2\} \qquad \text{and} \qquad \XMulti \coloneqq \XTraces{1} \times \XTraces{2},
  \end{aligned}
\end{equation}
and equip these spaces with the graph norm:
\begin{align*}
  \norm{\vectorValued \phi_i}_{\XTraces{i}} & \coloneqq \left(\norm{\phi_{i,\dataDirichlet}}_{H^{1/2} \left( \Gamma_i \right)}^2 + \norm{\phi_{i,\dataNeumann}}_{H^{-1/2} \left( \Gamma_i \right)}^2 \right)^{1/2} & \text{for } \vectorValued \phi_i = \left(\phi_{i,\dataDirichlet}, \phi_{i,\dataNeumann} \right)  \in \XTraces{i},\\
  \norm{\vectorValued \phi}_{\XMulti} & \coloneqq \sqrt{\norm{\vectorValued \phi_1}_{\XTraces{1}}^2 + \norm{\vectorValued \phi_2}_{\XTraces{2}}^2} & \text{for } \vectorValued{\phi} = \left( \vectorValued{\phi}_1, \vectorValued{\phi}_2 \right) \in \XMulti.
\end{align*}

The \textit{single trace space} is a subspace of $\XMulti$ and defined by
\begin{equation}
  \XSingle \coloneqq \left\{ \left(
      \begin{pmatrix}
        \phi_{i,\dataDirichlet}\\
        \phi_{i,\dataNeumann}
      \end{pmatrix}
    \right)_{i=1,2} \in \XMulti \mid \exists
\begin{pmatrix}
v\in H^{1}\left(  \mathbb{R}^{3}\right) \\
\mathbf{w}\in\vectorValued{H}\left(  \mathbb{R}^{3},\divergence\right)
\end{pmatrix}
,\forall i=1,2%
\begin{pmatrix}
\phi_{i,\dataDirichlet}=\gamma_{\dataDirichlet;i}v\\
\phi_{i,\dataNeumann}=\gamma_{\normalN;i}\mathbf{w}%
\end{pmatrix}
\right\},  \label{DefXsinglespace}%
\end{equation}
where the components of $\vectorValued \phi_i$ are denoted by $\phi_{i,\dataDirichlet}$, $\phi_{i,\dataNeumann}$;
the space $\vectorValued{H}\left(  \mathbb{R}^{3},\divergence\right)$ is defined e.g., in \cite[p.~26]{Girault86}.

The corresponding Cauchy trace operators are given by
\[
\boldsymbol{\gamma}_{i}:H\left(  \Delta,\Omega_{i}\right)  \rightarrow
\boldsymbol{X}_{i},\quad\boldsymbol{\gamma}_{i}(v)=\left(  \gamma
_{\dataDirichlet,i}v,\gamma_{\dataNeumann,i}v\right)  ^{\top},
\]%
\[
\traceFull=\left(  \boldsymbol{\gamma}_{1}%
,\boldsymbol{\gamma}_{2}\right)  :H\left(  \Delta,\Omega_{1}\right)  \times
H\left(  \Delta,\Omega_{2}\right)  \rightarrow\boldsymbol{X}%
^{\operatorname{mult}}.
\]
It is known from \cite[Lem. 3.5]{CostabelElemRes} that the range of
$\boldsymbol{\gamma}_{i}$ is dense in $\boldsymbol{X}_{i}$. Since the spaces
$H^{1/2}\left(  \Gamma_{i}\right)  $ and $H^{-1/2}\left(  \Gamma_{i}\right)  $
are dual to each other, we have that the Cauchy trace spaces are in
self-duality with respect to the symmetric dual pairing $\left\langle
\cdot,\cdot\right\rangle _{\Gamma_{i}}^{+}$.

In the context of the wave equation, these (spatial) trace spaces are
considered as spaces of values of time-depending functions (distributions). To
define the relevant function space we first consider the Schwartz class
\[
\mathfrak{S}\left(  \mathbb{R}\right)  \coloneqq\left\{  \varphi\in C^{\infty
}\left(  \mathbb{R}\right)  \mid\forall k\in\mathbb{N}_{0},\quad\forall
p\in\mathbb{P}\left(  \mathbb{R}\right)  :\quad p\varphi^{\left(  k\right)
}\in L^{\infty}\left(  \mathbb{R}\right)  \right\}  ,
\]
where $\mathbb{P}\left(  \mathbb{R}\right)  $ denotes the space of polynomials
(with complex coefficients). $\mathfrak{S}\left(  \mathbb{R}\right)  $ can be
equipped with a metric that makes this space complete. A tempered distribution
with values in a Banach space $X$ is a continuous linear map $f:\mathfrak{S}%
\left(  \mathbb{R}\right)  \rightarrow X$. A \textit{causal tempered distribution with
values in }$X$ is a tempered $X$-valued distribution such that%
\begin{equation*}
  f (\varphi) = 0 \qquad \forall \varphi \in \mathfrak{S} (\mathbb{R}) \text{ such that } \operatorname*{supp} \varphi \subset \intervalOO{-\infty,0},
\end{equation*}
and following the notation in \cite{sayas_book} we write
\begin{equation*}
  f \in \operatorname*{CT} (X), \quad \operatorname*{CT}(X): \text{ space of causal tempered distributions with values in } X.
\end{equation*}

\begin{definition}
The space\footnote{$\operatorname*{TD}$ for \textquotedblleft time
domain\textquotedblright.} $\operatorname*{TD}\left(  X\right)  $ consists of
all (possibly distributional) derivatives of continuous causal $X$-valued
functions with, at most, polynomial growth.
\end{definition}

We employ the \emph{direct method} to transform the wave equation into a
space-time boundary integral equation and start with the Kirchhoff
representation formula. The key potential is given by
\[
\left(  \mathcal{G}_{i}\boldsymbol{\phi}\right)  (t,x)\coloneqq\int_{0}%
^{t}\left\langle \boldsymbol{\gamma}_{i}k_{i}(t-\tau,x-\cdot),\boldsymbol{\phi
}(\tau)\right\rangle _{\Gamma_{i}}^{-}\mathop{}\!\mathrm{d}\tau
\]
for $\boldsymbol{\phi}\in\operatorname*{TD}\left(  \boldsymbol{X}_{i}\right)
$ and $k_{i}$ as in (\ref{eq:fundamentalSolution}).

Then, every $u_i \in \operatorname*{TD} \left( H^1 \left( \Delta,\Omega_i \right) \right)$ that satisfies $p_i^2 \partial_{t}^2 u_i - a_i^2 \Delta u_i = 0$ and $u_i(0) = \partial_{t} u_i(0) = 0$ also satisfies the representation formula
(see \cite[Prop.~3.5.1]{sayas_book})
\[
u_{i}=\mathcal{G}_{i}\ast\boldsymbol{\gamma}_{i}u_{i}.
\]
We introduce the \emph{Calder\'{o}n projector} $\boldsymbol{\mathcal{P}}%
_{i}(t):\boldsymbol{X}_{i}\rightarrow\boldsymbol{X}_{i}$ by
\[
\boldsymbol{\mathcal{P}}_{i}(t)\coloneqq\boldsymbol{\gamma}_{i}\mathcal{G}%
_{i}(t).
\]
$u_{i}\in\operatorname*{TD}\left(  H^{1}\left(
\Delta,\Omega_{i}\right)  \right)  $ solves the homogeneous wave equation
$p_{i}^{2}\partial_{t}^{2}u_{i}-a_{i}^{2}\Delta u_{i}=0$ in $\Omega_{i}$ and
$u_{i}(0)=\partial_{t}u_{i}(0)=0$, if and only if
(\cite[Section~3.5]{sayas_book})
\begin{equation}
\left(  \boldsymbol{\mathcal{P}}_{i}(\cdot)-\delta_{0}\right)  \ast
\boldsymbol{\gamma}_{i}u_{i}(\cdot)=0.\label{CaldProjEq}%
\end{equation}
This equation will be our starting point for the formulation of problem
(\ref{trans_probl}) as a system of integral equations. Next we transform this
equation to the Laplace domain; cf. Remark \ref{RemLaplTD}. The Laplace
transform of (\ref{CaldProjEq}) is given by
\begin{equation}
\left(  \boldsymbol{\mathsf{P}}_{i}(s)-\boldsymbol{Id}\right)
\boldsymbol{\gamma}_{i}\hat{u}_{i}(s)=0,\label{caldprojLapl}%
\end{equation}
where $\boldsymbol{Id}$ denotes the identity operator and
\begin{equation*}
  \vectorValued{\laplaceOperator{P}}_i \coloneqq \trace{i} \laplaceOperator G_i \quad \text{with} \quad
  \laplaceOperator G_i (\vectorValued{\laplaceFunction{\phi}}(s);s,x) \coloneqq \left\langle \trace{i} \laplaceFunction{k}_i (s, x - \cdot), \vectorValued{\laplaceFunction{\phi}}(s) \right\rangle_{\Gamma_i}^- \text{ for }
  \vectorValued{\laplaceFunction{\phi}}(s) \in \XTraces{i}, s \in \mathbb C_{\sigma_0}.
\end{equation*}
The operator $\boldsymbol{\mathsf{P}}_{i}(s)-\frac{\boldsymbol{Id}%
}{2}$ is denoted as the \emph{Calderón operator}. It turns out, that this
operator is not optimally scaled in terms of the frequency variable $s$ for
its stability analysis. We employ a further transformation and introduce the
frequency dependent diagonal matrix and frequency-weighted trace operators%
\begin{equation*}
  \vectorValued{\laplaceOperator{D}}(s) \coloneqq \diag \left[ s^{1/2}, s^{-1/2} \right] \in\mathbb{C}^{2\times2};
  \traceScaled{i} \coloneqq \vectorValued{\laplaceOperator{D}}(s) \trace{i} = \transpose{( s^{1/2} \traceD{i}, s^{-1/2}\traceN{i})} \text{ and }
  \traceFullScaled \coloneqq \diag\left( \vectorValued{\laplaceOperator{D}}(s), \vectorValued{\laplaceOperator{D}}(s) \right)  \traceFull.
\end{equation*}
This allows us to define the scaled version of the block Calder\'{o}n operator
$\boldsymbol{\mathsf{A}}(s)\coloneqq\diag(\boldsymbol{\mathsf{A}}%
_{1}(s),\boldsymbol{\mathsf{A}}_{2}(s))$, with%
\begin{equation}
\boldsymbol{\mathsf{A}}_{i}(s)\coloneqq\boldsymbol{\mathsf{D}}(s)\left(
\boldsymbol{\mathsf{P}}_{i}(s)-\frac{\boldsymbol{Id}}{2}\right)
\boldsymbol{\mathsf{D}}^{-1}(s)\coloneqq%
\begin{bmatrix}
-\mathsf{K}_{i}(s) & s\mathsf{V}_{i}(s)\\
\frac{1}{s}\mathsf{W}_{i}(s) & \mathsf{K}_{i}^{\prime}(s)
\end{bmatrix}
\quad\text{for }i=1,2.\label{defAsj}%
\end{equation}
Then, (\ref{caldprojLapl}) can be written in the form
\[
\left(  \boldsymbol{\mathsf{A}}(s)-\frac{\boldsymbol{Id}}{2}\right)
\traceFullScaled\hat{u}=0.
\]

We will also need the following assumption on $\timeOperator T$.
\begin{assumption}
  \label{A1}
  The operator $\timeOperator{T}(t)$ in (\ref{trans_probl:I}) and (\ref{eq:urelations}) is the inverse Laplace transform of a bounded linear \emph{transfer operator} $\laplaceOperator{T}(s) : \tilde H_{\dataDirichlet}^{1/2} \left(\Gamma_{\dataImpedance}\right) \to \tilde H_{\dataNeumann}^{-1/2} \left(\Gamma_{\dataImpedance} \right)$ depending analytically on $s\in\mathbb{C}_{0}$, more precisely
  \begin{equation*}
    (\timeOperator{T} \ast \varphi)(t) = \mathcal{L}^{-1} \left(\laplaceOperator{T} \hat{\varphi}\right)(s)
  \end{equation*}
  for any function $\varphi \in \operatorname*{TD} \left(\tilde H_{\dataDirichlet}^{1/2}(\Gamma_{\dataImpedance})\right)$;
  $\laplaceOperator{T}(s)$ satisfies the following (dissipative) sign property:
  \begin{equation}
    \real \left\langle \laplaceOperator{T}(s) \laplaceFunction{\varphi} , \overline{\laplaceFunction{\varphi}} \right\rangle_{\Gamma_{\dataImpedance}} \leq 0 \quad
    \forall \laplaceFunction{\varphi} \in \tilde H_{\dataDirichlet}^{1/2}(\Gamma_{\dataImpedance})
.\label{signcond}%
\end{equation}
\end{assumption}

The following duality holds (the proof is a slight generalization of the well-known duality of $\tilde{H}^{1/2}(\Gamma)$ and $H^{-1/2}(\Gamma)$, which can be found e.g. in \cite[Theorem~3.14]{Mclean00})
\begin{equation*}
  \tilde{H}_{\dataDirichlet}^{\pm 1/2}(\Gamma_{\dataImpedance}) = \left(\tilde{H}_{\dataNeumann}^{\mp 1/2}(\Gamma_{\dataImpedance}) \right)'.
\end{equation*}
\begin{remark}
  If the transfer operator $\laplaceOperator{T}$, in the case of impedance boundary condition, is given by minus identity, $\laplaceOperator{T} = - \identity$, then Assumption~\ref{A1} is satisfied trivially since
  $\tilde H^{1/2}_{\dataDirichlet} \left( \Gamma_{\dataImpedance} \right)  \subseteq L^{2} \left( \Gamma_{\dataImpedance} \right) \subseteq \tilde H^{-1/2}_{\dataNeumann} \left( \Gamma_{\dataImpedance} \right)$.

  If $\laplaceOperator{T}_0(s)$ denotes the standard $\operatorname*{DtN}$ operator on $\partial\Omega$, one could define $\laplaceOperator{T}(s) \coloneqq Z^{\prime} \laplaceOperator{T}_0(s) Z$, where $Z:\tilde H_{\dataDirichlet}^{1/2}\left( \Gamma_{\dataImpedance} \right) \to H^{1/2} \left( \Gamma_0 \right) $ is a linear and bounded extension operator, e.g., the minimal $H^{1/2} \left( \Gamma_0 \right)$ extension and the projection $Z' \colon H^{-1/2} \left(\Gamma_0 \right) \to \tilde H_{\dataNeumann}^{-1/2} \left( \Gamma_{\dataImpedance} \right)$ is its dual.
  The sign condition then is inherited from the well-known sign property (see \cite[Eq.~(2.6.93)]{Nedelec01}) of $\laplaceOperator{T}_0(s)$ via
  \begin{equation*}
    \real \left\langle \laplaceOperator{T}(s) \laplaceFunction{\varphi}, \laplaceFunction{\varphi} \right\rangle
_{\Gamma_{\dataImpedance}} = \real \left\langle \mathsf{T} _{0} \laplaceFunction{\varphi
}^{\operatorname{ext}}_{0}, \laplaceFunction{\varphi}^{\operatorname{ext}}_{0}
\right\rangle _{\Gamma_{0}} \leq0 \quad\forall\laplaceFunction{\varphi} \in
H_{\dataDirichlet}^{1/2} \left(  \Gamma_{\dataImpedance} \right) ,
\end{equation*}
where $\laplaceFunction{\varphi}^{\operatorname{ext}}_{0}$ denotes the extension of
$\laplaceFunction{\varphi} \in\tilde H^{1/2}_{D} \left(  \Gamma_{\dataImpedance} \right) $ to
$\Gamma_{0}$ by zero.
\end{remark}

To deal with problem (\ref{trans_probl}) we incorporate Dirichlet and Neumann boundary conditions into the space $\mathbf{X}^{\operatorname{single}}$.
For this we extend the Dirichlet part $\Gamma_{\dataDirichlet}$ to a closed boundary (see Fig.~\ref{domain}) of a bounded domain $\Omega_{Z_{\dataDirichlet}}\subset\Omega_0$ (i.e., $\Omega_{Z_{\dataDirichlet}}$ lies outside the domain $\Omega$ where the problem is defined) such that $\partial \Omega_{Z_{\dataDirichlet}} \cap \Sigma = \Gamma_{\dataDirichlet}$.
We extend the Neumann part $\Gamma_{\dataNeumann}$ in the same way and obtain $\Omega_{Z_{\dataNeumann}}$.
Then we set

\begin{subequations}
  \label{eq:PartialZeroExtension}
  \begin{align}
    H_{\dataDirichlet}^1 \left( \mathbb R^3 \right)                        & \coloneqq \left\{ v \in H^1 \left( \mathbb R^3 \setminus \overline{\Omega_{Z_{\dataDirichlet}}} \right) \mid \restrict{v}_{\partial \Omega_{Z_{\dataDirichlet}}} = 0 \right\}, \label{eq:PartialZeroExtension:d} \\
    \vectorValued H_{\dataNeumann} \left( \mathbb R^3, \divergence \right) & \coloneqq \left\{ \vectorValued w \in \vectorValued H \left( \mathbb{R}^3 \setminus \overline{\Omega_{Z_{\dataNeumann}}}, \divergence \right) \mid  \left\langle \normalN_{{Z_{\dataNeumann}}}, \restrict{\vectorValued w}_{\partial \Omega_{Z_{\dataNeumann}}}\right\rangle = 0 \right\}, \label{eq:PartialZeroExtension:n}
  \end{align}
\end{subequations}
and define the space of Cauchy traces of global fields whose Dirichlet and Neumann components vanish on $\Gamma_{\dataDirichlet}$ and $\Gamma_{\dataNeumann}$ respectively;
this space naturally arises when offsetting Cauchy traces with the boundary data (see Section~\ref{sec:dirichlet}).

\begin{equation}
  \label{DefXsinglespaceD}
  \XSingleZero \coloneqq \left\{ \left(
      \begin{pmatrix}
        \phi_{i,\dataDirichlet}\\
        \phi_{i,\dataNeumann}
      \end{pmatrix}
    \right)_{i=1,2} \in \XSingle \mid \exists
    \begin{pmatrix}
      v \in H_{\dataDirichlet}^{1} \left( \mathbb{R}^{3} \right) \\
      \vectorValued{w} \in \vectorValued{H}_{\dataNeumann} \left( \mathbb{R}^{3},\divergence \right)
    \end{pmatrix}
    \forall i=1,2:
    \begin{pmatrix}
      \phi_{i,\dataDirichlet} = \traceD{i} v\\
      \phi_{i,\dataNeumann} = \traceNormal{i} \vectorValued{w}
    \end{pmatrix}
  \right\}.
\end{equation}

\subsection{Treatment of the Neumann and Dirichlet boundary conditions\label{sec:dirichlet}}

Now the transmission conditions (\ref{trans_probl:J}) are built into the function space $\mathbf{X}^{\operatorname{single}}$; we take into account the boundary conditions on $\Gamma_{\dataDirichlet}$ and $\Gamma_{\dataNeumann}$ next.

To obtain a variational formulation for the unknown Cauchy data of the transmission problem (\ref{trans_probl}) with balanced test and trial spaces we consider an offset function $\laplaceFunction b = \laplaceFunction b(s) \in H^1(\Delta,\Omega)$ such that
\begin{equation*}
  \left( \traceDFull \laplaceFunction b \right)|_{\Gamma_{\dataDirichlet}} = \laplaceFunction g_{\dataDirichlet}, \quad
  \left(\traceNFull \laplaceFunction b \right)|_{\Gamma_{\dataNeumann}}   = \laplaceFunction d_{\dataNeumann}.
\end{equation*}
In the simplest case, the function $\laplaceFunction{b} \in H^1(\Delta,\Omega)$ is \emph{given} and the boundary data $\timeFunction g_{\dataDirichlet}, \timeFunction d_{\dataNeumann}$ in the problem formulation (\ref{trans_probl:D},\ref{trans_probl:N}) were obtained from $\laplaceFunction{b}$;
in this case an immediate extension to $\Sigma$ is available.
If $\laplaceFunction{b}$ is not given, it can be \emph{computed} as the solution of a well-posed boundary value problem for $-\Delta$ with mixed boundary conditions. 
We emphasize that, as far as the boundary problem is concerned, only the traces of $\laplaceFunction{b}$ are required.

We set
\begin{equation*}
  \uZero \coloneqq u - b \text{ with } b \coloneqq \mathcal{L}^{-1} \laplaceFunction{b}
\end{equation*}
and observe that $\traceFull(s)(\hat{u}(s)-\hat{b}(s))=\traceFull(s) \zero{\laplaceFunction{u}}\in \XSingleZero$.

The boundary conditions (\ref{trans_probl}) for the new function $\uZero$ now read
\begin{equation}
  \label{eq:urelations}
  \traceDFull \uZero|_{\Gamma_{\dataDirichlet}}=0, \qquad
  \traceNFull \uZero|_{\Gamma_{\dataNeumann}}=0, \qquad
  \traceNFull \uZero|_{\Gamma_{\dataImpedance}} - \timeOperator T \ast (\traceDFull \dot u_0)|_{\Gamma_{\dataImpedance}} = d_{\dataImpedance} + \traceNFull b|_{\Gamma_{\dataImpedance}} +  \timeOperator T \ast \left( \traceDFull \dot b \right)|_{\Gamma_{\dataImpedance}}.
\end{equation}
Note that the expression $\timeOperator{T} \ast (\traceDFull \zero{\dot u})|_{\Gamma_{\dataImpedance}}$ is well defined, because our assumptions on $g_{\dataDirichlet}$ imply that $b|_{\Gamma_{\dataImpedance}}$ can be extended by zero on $\Gamma_{\dataDirichlet}$.

Since $\zero{\dot u}$ vanishes on $\Gamma_{\dataDirichlet}$ and $\frac{\partial \uZero}{\partial n}$ vanishes on $\Gamma_{\dataNeumann}$, the function $\restrict{\traceFullScaled \zero{\laplaceFunction u}}_{\Gamma_{\dataImpedance}}$ belongs to $\tilde {H}_{\dataDirichlet}^{1/2}\left(  \Gamma_{\dataImpedance}\right)  \times \tilde{H}_{\dataNeumann}^{-1/2}\left(  \Gamma_{\dataImpedance}\right)$;
(see (\ref{eq:halfTilde})).

Let $\vectorValued{\Phi} = \transpose{(\phi_{1,\dataDirichlet}, \phi_{1,\dataNeumann},\phi_{2,\dataDirichlet}, \phi_{2,\dataNeumann})} , \vectorValued{\Psi} = \transpose{(\psi_{1,\dataDirichlet}, \psi_{1,\dataNeumann},\psi_{2,\dataDirichlet}, \psi_{2,\dataNeumann})} \in \XMulti$.
In analogy to (\ref{eq:pairings}), we define the pairing on $\Sigma$:
\begin{equation*}
  \left\langle \vectorValued{\Phi} , \vectorValued{\Psi} \right\rangle^+_{\Sigma}
  \coloneqq \sum_{j=1}^{2} \left( \left\langle \phi_{j,\dataDirichlet}, \psi_{j,\dataNeumann} \right\rangle_{\Gamma_j} + \left\langle \psi_{j,\dataDirichlet},\phi_{j,\dataNeumann} \right\rangle_{\Gamma_j} \right)
\end{equation*}
and on the open surface $\Gamma_{\dataImpedance}$:
\begin{equation*}
  \left\langle \vectorValued{\Phi} , \vectorValued{\Psi} \right\rangle^+_{\Gamma_{\dataImpedance}}
  \coloneqq \sum_{j=1}^{2} \left( \left\langle \phi_{j,\dataDirichlet}, \psi_{j,\dataNeumann} \right\rangle_{\Gamma_{j,\dataImpedance}} + \left\langle \psi_{j,\dataDirichlet},\phi_{j,\dataNeumann}\right\rangle _{\Gamma_{j,\dataImpedance}} \right).
\end{equation*}

\begin{proposition}
  \label{th:bilinearEquivalent}
  For any $\vectorValued{\Phi}, \vectorValued{\Psi} \in \XSingleZero$, $\left\langle \vectorValued{\Phi} , \vectorValued{\Psi} \right\rangle^+_{\Sigma} = \left\langle \vectorValued{\Phi} , \vectorValued{\Psi} \right\rangle^+_{\Gamma_{\dataImpedance}}$.
\end{proposition}
\begin{proof}
  Fix $j \in \{1,2\}$;
  since $\restrict{\phi_{j,\dataDirichlet}}_{\Gamma_{\dataDirichlet}} = 0$ and $\restrict{\phi_{j,\dataNeumann}}_{\Gamma_{\dataNeumann}} = 0$ (and the same properties hold for $\vectorValued{\Psi}$) we get
  \begin{equation*}
    \left\langle \phi_{j,\dataDirichlet}, \psi_{j,\dataNeumann} \right\rangle_{\Gamma_j} =
    \left\langle \phi_{j,\dataDirichlet}, \psi_{j,\dataNeumann} \right\rangle_{\Gamma_{j,\dataImpedance}} +
    \left\langle \phi_{j,\dataDirichlet}, \psi_{j,\dataNeumann} \right\rangle_{\Gamma_{j,\dataJump}}.
  \end{equation*}
  Moreover since $\vectorValued{\Phi} \in \XSingle$, $\restrict{\phi_{1,\dataDirichlet}}_{\Gamma_{\dataJump}} = \restrict{\phi_{2,\dataDirichlet}}_{\Gamma_{\dataJump}}$ and $\restrict{\phi_{1,\dataNeumann}}_{\Gamma_{\dataJump}} = - \restrict{\phi_{2,\dataNeumann}}_{\Gamma_{\dataJump}}$ (and the same properties hold for $\vectorValued{\Psi}$).
  Hence
  \begin{equation*}
    \begin{aligned}
      \left\langle \vectorValued{\Phi}, \vectorValued{\Psi} \right\rangle^+_{\Sigma} = & \sum_{j=1}^2 \left( \left\langle \phi_{j,\dataDirichlet}, \psi_{j,\dataNeumann} \right\rangle_{\Gamma_j} + \left\langle \psi_{j,\dataDirichlet},\phi_{j,\dataNeumann} \right\rangle_{\Gamma_j} \right)\\
      = & \sum_{j=1}^2  \left( \left\langle \phi_{j,\dataDirichlet}, \psi_{j,\dataNeumann} \right\rangle_{\Gamma_{j,\dataImpedance}} + \left\langle \psi_{j,\dataDirichlet}, \phi_{j,\dataNeumann} \right\rangle_{\Gamma_{j,\dataImpedance}} \right)\\
      & + \left\langle \phi_{1,\dataDirichlet}, \psi_{1,\dataNeumann} \right\rangle_{\Gamma_{1,\dataJump}} +\left\langle \phi_{2,\dataDirichlet}, \psi_{2,\dataNeumann} \right\rangle_{\Gamma_{2,\dataJump}}\\
      & + \left\langle \psi_{1,\dataDirichlet}, \phi_{1,\dataNeumann} \right\rangle_{\Gamma_{1,\dataJump}} +\left\langle \psi_{2,\dataDirichlet}, \phi_{2,\dataNeumann} \right\rangle_{\Gamma_{2,\dataJump}} = \left\langle \vectorValued{\Phi} , \vectorValued{\Psi} \right\rangle^+_{\Gamma_{\dataImpedance}}.
    \end{aligned}
  \end{equation*}
  These two pairings therefore coincide on $\XSingleZero$.
\end{proof}

Define for $\vectorValued{\Phi}, \vectorValued{\Psi} \in \XMulti$:
\begin{subequations}
  \label{eq:a0l0}
  \begin{align}
    \aZero \left( s; \vectorValued{\Phi}, \vectorValued{\Psi} \right) \coloneqq & \left\langle \left( \vectorValued{\laplaceOperator A}(s) - \frac{\vectorValued{\identity}}{2} \right) \vectorValued{\Phi}, \conjugate{\vectorValued{\Psi}} \right\rangle_{\Sigma}^+, \label{eq:a0}\\
    \lZero \left( s; \vectorValued{\Psi} \right) \coloneqq & \aZero (s; \traceFullScaled \laplaceFunction{b}, \vectorValued{\Psi}). \label{eq:l0}
  \end{align}
\end{subequations}

\begin{problem}
  Find the Laplace transformed Cauchy traces $\traceFullScaled \zero{\laplaceFunction u} \in \XSingleZero$
  \begin{subequations}
    \label{eq:variationalZero}
    \begin{align}
      \aZero(s; \traceFullScaled \zero{\laplaceFunction u}, \vectorValued{\laplaceFunction \Psi}(s)) = & - \lZero \left( s;\vectorValued{\laplaceFunction \Psi}(s) \right) \quad
                                                                                                         \forall \vectorValued{\laplaceFunction \Psi}(s) \in \XSingleZero, s \in \mathbb C_{\sigma_0};\label{eq:variationalZero:variational}\\
      \traceNFull \zero{\laplaceFunction u}|_{\Gamma_{\dataImpedance}} - \laplaceOperator T \left( \traceDFull s \zero{\laplaceFunction u} \right)|_{\Gamma_{\dataImpedance}} = &
        \laplaceFunction d_{\dataImpedance} - \traceNFull \laplaceFunction b|_{\Gamma_{\dataImpedance}} +  \laplaceOperator T \left( \traceDFull s \laplaceFunction b \right)|_{\Gamma_{\dataImpedance}},\label{eq:variationalZero:boundary}
    \end{align}
  \end{subequations}
  where the second equation expresses the boundary condition on $\Gamma_{\dataImpedance}$, which will be incorporated in the variational formulation in Section~\ref{SecLinComb}.
\end{problem}

\subsection{Variational formulation including impedance boundary conditions \label{SecLinComb}}
Finally, we incorporate the impedance boundary condition (\ref{trans_probl:I}).

We start by defining, for $\boldsymbol{\hat{\phi}}\in \XSingleZero$, functions $\hat{\phi}_{\dataDirichlet},\hat
{\phi}_{\dataNeumann}$ on $\Gamma_{\dataImpedance}\subseteq\partial\Omega$ such that
\begin{align*}
  \laplaceFunction{\phi}_{\dataDirichlet}|_{\Gamma_{j,\dataImpedance}}\coloneqq & \laplaceFunction{\phi}_{j,\dataDirichlet}|_{\Gamma_{j,\dataImpedance}}, &
  \laplaceFunction{\phi}_{\dataNeumann}|_{\Gamma_{j,\dataImpedance}}\coloneqq & \laplaceFunction{\phi}_{j,\dataNeumann}|_{\Gamma_{j,\dataImpedance}}, & \text{for } j=1,2.
\end{align*}
Due to the definition of $\XSingleZero$ we have
$\hat{\phi}_{\dataDirichlet}\in\tilde{H}_{\dataDirichlet}^{1/2}(\Gamma
_{\dataImpedance}),\hat{\phi}_{\dataNeumann}\in\tilde{H}_{\dataNeumann}%
^{-1/2}(\Gamma_{\dataImpedance})$.

We treat the Dirichlet and Neumann boundary condition as explained in Section \ref{sec:dirichlet}, but incorporate the impedance condition of (\ref{eq:urelations}) keeping both the Dirichlet and Neumann trace as unknowns in the resulting equations.
Recall the impedance condition (cf. (\ref{eq:urelations})), and set $\laplaceFunction{\xi} := \traceFullScaled \zero{\laplaceFunction{u}}$:
\begin{equation*}
  \laplaceFunction{\xi}_{\dataNeumann}(s) - \mathsf{T} (s) \laplaceFunction{\xi}_{\dataDirichlet}(s) = s^{-1/2} \hat{d}_{\dataImpedance}(s) - s^{-1/2} (\gamma_{\dataNeumann, 0} \hat b(s))|_{\Gamma_{\dataImpedance}} + s^{1/2} \mathsf{T}(s) (\gamma_{\dataDirichlet, 0} \hat b(s))|_{\Gamma_{\dataImpedance}}.
\end{equation*}
This gives rise to the definition of the sesquilinear form $\aImp(s) \colon \XSingleZero \times \XSingleZero \to\mathbb{C}$ and right-hand side functional $\lImp (s) \colon \XSingleZero \to \mathbb{C}$:
\begin{align*}
  \aImp (s; \vectorValued{\phi}, \vectorValued{\psi}) & \coloneqq \left\langle \phi_{\dataNeumann} - \laplaceOperator T(s) \phi_{\dataDirichlet}, \conjugate{\psi_{\dataDirichlet}} \right\rangle_{\Gamma_{\dataImpedance}}, \\
  \lImp(s; \vectorValued{\psi})                       & \coloneqq \left\langle s^{-1/2} \laplaceFunction d_{\dataImpedance}(s) - s^{-1/2} \restrict{\traceNFull \laplaceFunction b(s)}_{\Gamma_{\dataImpedance}} + s^{1/2} \laplaceOperator T(s) \restrict{\traceDFull \hat b(s)}_{\Gamma_{\dataImpedance}}, \conjugate{\psi_{\dataDirichlet}} \right\rangle_{\Gamma_{\dataImpedance}}.
\end{align*}

\begin{problem}[Mixed Formulation of Acoustic Mixed Transmission Problem]
  \label{problem:mixedFormulation}
  Find $\vectorValued{\laplaceFunction \phi} \in \XSingleZero$ such that
  \begin{equation*}
    \aMix \left(s; \vectorValued{\laplaceFunction \phi}, \vectorValued{\laplaceFunction \psi} \right) = \lMix \left(s; \vectorValued{\laplaceFunction \psi} \right) \quad\forall \vectorValued{\laplaceFunction \psi} \in \XSingleZero,
  \end{equation*}
  where $\aMix(s) \coloneqq \aZero(s) + \aImp(s)$ and $\lMix(s) \coloneqq \lZero(s) + \lImp(s)$.
\end{problem}

The corresponding formulation in the time domain is the following.
\addtocounter{theorem}{-1}
\begin{problem}[Time domain formulation of Acoustic Mixed Transmission Problem]
  For any $t \in \intervalCC{0,T}$, find $\vectorValued{\timeFunction{\phi}} \in \XSingleZero$ such that
  \begin{equation}
    \label{eq:variationalTime:Tranmission_direct_form_ext}
    \begin{aligned}
      & \left\langle \left( \vectorValued{\timeOperator{A}}(t) - \frac{\vectorValued{\delta_{0}}}{2} \right) * \vectorValued{\timeFunction{\phi}}(t), \overline{\vectorValued{\timeFunction{\psi}}} \right\rangle_{\Sigma}^{+} +
      \left\langle \timeFunction{\phi}_{\dataNeumann} - \timeOperator{T}(t) * \timeFunction{\phi}_{\dataDirichlet}(t), \overline{\timeFunction{\psi}_{\dataDirichlet}} \right\rangle_{\Gamma_{\dataImpedance}} = \\
      & \qquad \left\langle \left( \vectorValued{\timeOperator{A}}(t) - \frac{\vectorValued{\delta_{0}}}{2} \right) * \traceFull (t) b(t), \overline{\vectorValued{\timeFunction{\psi}}} \right\rangle_{\Sigma}^{+} +
      \left\langle \partial_t^{-1/2} \timeFunction d_{\dataImpedance}(t) - \partial_{t}^{-1/2} \restrict{\traceNFull \laplaceFunction b(t)}_{\Gamma_{\dataImpedance}} + \timeOperator T(t) * \partial_{t}^{1/2} \restrict{\traceDFull \timeFunction b(t)}_{\Gamma_{\dataImpedance}}, \conjugate{\psi_{\dataDirichlet}} \right\rangle_{\Gamma_{\dataImpedance}} \\
      & \qquad \forall \vectorValued{\timeFunction{\psi}} \in \XSingleZero,
    \end{aligned}
  \end{equation}
  where $\traceFull (t) \coloneqq \transpose{\left( \partial_{t}^{1/2} \traceD{1}, \partial_{t}^{-1/2} \traceN{1}, \partial_{t}^{1/2} \traceD{2}, \partial_{t}^{-1/2} \traceN{2} \right)}$,
  \begin{equation*}
    \vectorValued{\timeOperator{A}}(t) \coloneqq
    \begin{bmatrix}
      \begin{matrix}
        -\timeOperator K_1(t) & \partial_{t} \timeOperator V_1(t)\\
        \partial_{t}^{-1} \timeOperator W_1(t)  & \timeOperator K_1'(t)
      \end{matrix} &\\
      & \begin{matrix}
        -\timeOperator K_2(t) & \partial_{t} \timeOperator V_2(t)\\
        \partial_{t}^{-1} \timeOperator W_2(t)  & \timeOperator K_2'(t)
      \end{matrix}
    \end{bmatrix}
  \end{equation*}
  and the notation $\partial_{t}^{\mu}$ for $\mu \in \mathbb{R}$ is defined as the inverse Laplace transform applied to the multiplication by $s^{\mu}$, i.e., $\partial_{t}^{\mu} \phi \coloneqq \mathcal{L}^{-1} (s^{\mu} \phi )$.
  For $\mu = -1$, $\partial_{t}^{-1}$ is the antiderivative with respect to $t$: $\partial_{t}^{-1} \timeFunction{\phi}(t) = \int_0^t \timeFunction{\phi}(\tau) \de \tau$.
\end{problem}

The solution of this problem gives the trace $\traceFull(t) \zero{\timeFunction u}$.
The solution of the wave equation (\ref{trans_probl}) is then obtained in two steps: first the offset $\traceFull(t) \timeFunction b$ is added to obtain the solution for the boundary data (Section~\ref{sec:dirichlet});
then the solution in the whole domain can be obtained using the layer potentials (Section~\ref{SecLPBIE}).

\begin{remark}
  It is also possible to use (\ref{eq:urelations}) to eliminate the Neumann data on $\Gamma_{\dataImpedance}$.
  This would lead to a system of integral equations containing the minimal number of unknowns:
  the Neumann data on $\Gamma_{\dataDirichlet}$, the Dirichlet data on $\Gamma_{\dataNeumann} \cup\Gamma_{\dataImpedance}$, the Dirichlet and Neumann data on $\Gamma_{\dataJump}$.
  The drawback is that a function $\laplaceFunction d$ on $\Omega$ has to be constructed, which provides a skeleton extension of the impedance data; more precisely, $\laplaceFunction d$ must satisfy
  \begin{align*}[left= {\empheqlbrace}]
    - s^{-1/2} \restrict{\traceNFull \laplaceFunction d}_{\Gamma_{\dataImpedance}} - s^{1/2} \laplaceOperator T \restrict{\traceDFull \laplaceFunction d}_{\Gamma_{\dataImpedance}} &  =
                                                                                                                                                                                      s^{-1/2} \laplaceFunction d_{\dataImpedance} + s^{-1/2} \restrict{\traceNFull \laplaceFunction b}_{\Gamma_{\dataImpedance}} + s^{1/2} \laplaceOperator T \restrict{\traceDFull \laplaceFunction b}_{\Gamma_{\dataImpedance}},\\
    \left[ \hat d \right]_{\Gamma_{\dataJump}}=\left[ a^2 \frac{\partial \hat d}{\partial n} \right]_{\Gamma_{\dataJump}} & =0,\\
    \restrict{\traceNFull \laplaceFunction d}_{\Gamma_{\dataNeumann}} & = 0,\\
    \restrict{\traceDFull \laplaceFunction d}_{\Gamma_{\dataDirichlet}} & = 0.
  \end{align*}
\end{remark}

\section{Well-Posedness of Time Domain Boundary Integral Equation\label{SecCaldOp}}
\label{sec:WellPosednessAnalysis}

In the following we will recall mapping properties of the single and double layer potentials and their corresponding integral equations.

For $j\in\{0,1,2\}$, the proofs of the following propositions
(Prop.~\ref{PropEll} and the 3rd and 6th inequality in Prop.~\ref{PropCont},
(\ref{table_map_prop})) go back to \cite{bambduong}. We have used here the
estimates for the boundary integral operators as in \cite{LaSa2009}.

\begin{proposition}
\label{PropCont}Let $s\in\mathbb{C}_{\sigma_{0}}$ and recall (\ref{defatau}).
Then, for $j \in\{0,1,2\}$, the operators $\laplaceOperator S_j(s)$, $\laplaceOperator D_j(s)$, $\laplaceOperator V_j(s)$, $\laplaceOperator K_j(s)$, $\laplaceOperator K_j^{\prime}(s)$, $\laplaceOperator W_j(s)$, satisfy the following mapping properties:
for all
$\Phi\in H^{-1/2} \left(  \Gamma_{j} \right) $ and $\Psi\in H^{1/2} \left(
\Gamma_{j} \right) $ there is some constant $C$ independent of $s$ such that
\begin{equation}
\label{table_map_prop}\begin{aligned} \ensuremath{\mathsf{S}}_j(s) & :H^{-1/2}\left( \Gamma_j \right) \to H^1 \left(\mathbb R^3 \right), & \norm{\laplaceOperator S_i(s) \Phi}_{H^1 \left(\mathbb R^3 \right)} & \leq C \abs{s} \norm{\Phi}_{H^{-1/2}\left(\Gamma_j\right)}, \\ \ensuremath{\mathsf{D}}_j(s) & :H^{1/2} \left( \Gamma_j \right) \to H^1 \left(\mathbb R^3 \setminus \Gamma_j \right), & \norm{\laplaceOperator D_i(s) \Psi}_{H^1 \left(\mathbb R^3 \setminus \Gamma_j \right)} & \leq C \abs{s}^{3/2}\norm{\Psi}_{H^{1/2} \left(\Gamma_j\right)}, \\ \ensuremath{\mathsf{V}}_j(s) & :H^{-1/2}\left( \Gamma_j \right) \to H^{1/2} \left(\Gamma_j \right), & \norm{\laplaceOperator V_j(s) \Phi}_{H^{1/2} \left(\Gamma_j\right)} & \leq C \abs{s} \norm{\Phi}_{H^{-1/2}\left(\Gamma_j\right)}, \\ \ensuremath{\mathsf{K}}_j(s) & :H^{1/2} \left( \Gamma_j \right) \to H^{1/2} \left(\Gamma_j \right), & \norm{\laplaceOperator K_j(s) \Psi}_{H^{1/2} \left(\Gamma_j\right)} & \leq C \abs{s}^{3/2}\norm{\Psi}_{H^{1/2} \left(\Gamma_j\right)}, \\ \ensuremath{\mathsf{K}}_j'(s) & :H^{-1/2}\left( \Gamma_j \right) \to H^{-1/2}\left(\Gamma_j \right), & \norm{\laplaceOperator K_j'(s)\Phi}_{H^{-1/2}\left(\Gamma_j\right)} & \leq C \abs{s}^{3/2}\norm{\Phi}_{H^{-1/2}\left(\Gamma_j\right)}, \\ \ensuremath{\mathsf{W}}_j(s) & :H^{1/2} \left( \Gamma_j \right) \to H^{-1/2}\left(\Gamma_j \right), & \norm{\laplaceOperator W_j(s) \Psi}_{H^{-1/2}\left(\Gamma_j\right)} & \leq C \abs{s}^2 \norm{\Psi}_{H^{1/2} \left(\Gamma_j\right)}. \end{aligned}
\end{equation}
\end{proposition}

Next we analyse the operators $\vectorValued{\laplaceOperator{A}}_i (s)$ which appear (through $\vectorValued{\laplaceOperator{A}}(s)$) in the definition of the sesquilinear form (\ref{eq:a0})
\begin{proposition}
  \label{PropEll} Let $s \in \mathbb{C}_{\sigma_0}$ and recall (\ref{defatau}).
  Then, for $i \in \{1,2\}$, $\vectorValued{\laplaceOperator{A}}_i (s)$ defined in (\ref{defAsj}) satisfies the coercivity estimate
\[
\real\left\langle \boldsymbol{\mathsf{A}}_{i}(s)%
\begin{pmatrix}
\psi\\
\varphi
\end{pmatrix}
,%
\begin{pmatrix}
\overline{\psi}\\
\overline{\varphi}%
\end{pmatrix}
\right\rangle _{\Gamma_{i}}^{+}\geq\beta\min\left\{  1,\abs{s}^{2}\right\}
\frac{\real s}{\abs{s}^{2}}\left(  \norm{\varphi}_{H^{1/2}\left(  \Gamma
_{i}\right)  }^{2}+\norm{\psi}_{H^{-1/2}\left(  \Gamma_{i}\right)  }%
^{2}\right)  ,
\]
for all $\left(  \varphi,\psi\right)  \in H^{1/2}\left(  \Gamma_{i}\right)
\times H^{-1/2}\left(  \Gamma_{i}\right)  $, for some $\beta>0$ and for all
$s\in\mathbb{C}_{\sigma_{0}}$.
\end{proposition}

\begin{proof}
Fix $i\in\{1,2\}$. A straightforward calculation shows that
\[
\left\langle \boldsymbol{\mathsf{A} } _{i}(s)
\begin{pmatrix}
\varphi\\
\psi
\end{pmatrix}
,
\begin{pmatrix}
\overline{\kappa}\\
\overline{\rho}%
\end{pmatrix}
\right\rangle ^{+}_{\Gamma_{i}} = \left\langle
\begin{pmatrix}
\overline{\rho}\\
- \overline{\kappa}%
\end{pmatrix}
, \boldsymbol{\mathsf{B} } _{i}(s)
\begin{pmatrix}
\psi\\
- \varphi
\end{pmatrix}
\right\rangle _{\Gamma_{i}}%
\]
for $\mathbf{B}(s)\coloneqq
\begin{bmatrix}
s \mathsf{V} _{i} (s) & \mathsf{K} _{i} (s)\\
- \mathsf{K} _{i}^{\prime}(s) & \frac{1}{s} \mathsf{W} _{i} (s)
\end{bmatrix}
.$

This operator was analyzed in \cite[Lem. 3.1]{banjai_coupling}: it maps
$H^{-1/2} \left(  \Gamma_{i} \right)  \times H^{1/2} \left(  \Gamma_{i}
\right) $ continuously into $H^{1/2} \left(  \Gamma_{i} \right)  \times
H^{-1/2} \left(  \Gamma_{i} \right) $ and satisfies the coercivity estimate
\begin{equation*}
  \real \left\langle
    \begin{pmatrix}
      \overline{\psi}\\
      \overline{\varphi}
    \end{pmatrix}
    , \vectorValued{\laplaceOperator{B}}_i (\overline s)
    \begin{pmatrix}
      \psi\\
      \varphi
    \end{pmatrix}
  \right\rangle_{\Gamma_{i}} \geq\beta\min\left\{ 1,\abs{s}^{2} \right\} \frac{\real s}{\abs{s}^{2}} \left( \norm{\varphi}_{H^{1/2} (\Gamma_i)}^{2} + \norm{\psi}_{H^{-1/2} (\Gamma_{i})}^{2} \right),
\end{equation*}
for all $\left(  \varphi, \psi\right)  \in H^{1/2} \left(  \Gamma_{i} \right)
\times H^{-1/2} \left(  \Gamma_{i} \right) $.
\end{proof}

\begin{lemma}
\label{th:continuousLemma} The sesquilinear form $\left( \boldsymbol{\hat \phi} ,
\boldsymbol{\hat \psi} \right)  \mapsto\left\langle \boldsymbol{\mathsf{A} } (s)
\boldsymbol{\hat \phi} ,\overline{\boldsymbol{\hat \psi} } \right\rangle _{\Sigma
}^{+}$ is continuous and coercive: there exist constants $\eta,\zeta>0$,
possibly depending on $\sigma_{0}$ but not on $s \in\mathbb{C}_{\sigma_{0}}$
such that
\begin{align*}
  \abs{\left\langle \vectorValued{\laplaceOperator A}(s) \vectorValued{\hat{\phi}}, \overline{\vectorValued{\hat{\psi}}} \right\rangle_{\Sigma}^+} & \leq \eta \abs{s}^{2} \left( \norm{\vectorValued{\hat{\phi}}}_{\XMulti} \norm{\vectorValued{\hat{\psi}}}_{\XMulti} \right)  & \forall\boldsymbol{\hat{\phi}} ,\boldsymbol{\hat{\psi}}  & \in \XMulti, &  &\\
\real \left\langle \boldsymbol{\mathsf{A} } (s)\boldsymbol{\hat{\phi}} ,\overline{\boldsymbol{\hat{\phi}} }\right\rangle _{\Sigma}^{+}  &  \geq \zeta \frac{\real s}{\abs{s} ^{2}} \norm{\vectorValued{\hat{\phi}}}_{\XMulti}^{2} & \forall\boldsymbol{\hat{\phi}}  & \in\boldsymbol{X} ^{\operatorname{mult}} . &  &
\end{align*}
\end{lemma}

\begin{proof}
  We write $\norm{sV_j} $ short for the \emph{natural} operator norm, i.e.,
  $\norm{sV_j} = \norm{sV_j}_{H^{1/2} \left( \Gamma_j \right) \leftarrow H^{-1/2} \left( \Gamma_j \right)}$
  and apply this convention also for $\norm{K_j}$, $\norm{K_j'}$, $\norm{\frac{1}{s}W_j}$ denoting the natural operator norms according to the mapping properties listed in (\ref{table_map_prop}).
  We employ the mapping properties as in (\ref{table_map_prop}) and obtain, for any
  \begin{equation*}
    \vectorValued{\laplaceFunction \phi}=
    \begin{pmatrix}
      \laplaceFunction \phi_{1,\dataDirichlet}\\
      \laplaceFunction \phi_{1,\dataNeumann}\\
      \laplaceFunction \phi_{2,\dataDirichlet}\\
      \laplaceFunction \phi_{2,\dataNeumann}
    \end{pmatrix}
, \vectorValued{\laplaceFunction \psi}=
    \begin{pmatrix}
      \laplaceFunction \psi_{1,\dataDirichlet}\\
      \laplaceFunction \psi_{1,\dataNeumann}\\
      \laplaceFunction \psi_{2,\dataDirichlet}\\
      \laplaceFunction \psi_{2,\dataNeumann}
    \end{pmatrix}
    \in \XMulti
  \end{equation*}
  \begin{align*}
    \abs{\sum_{j=1}^2 \left\langle \vectorValued{\laplaceOperator A}_j(s)
    \begin{pmatrix}
      \laplaceFunction \phi_{j,\dataDirichlet}\\
      \laplaceFunction \phi_{j,\dataNeumann}
    \end{pmatrix},
    \begin{pmatrix}
      \conjugate{\laplaceFunction \psi}_{j,\dataDirichlet}\\
      \conjugate{\laplaceFunction \psi}_{j,\dataNeumann}
    \end{pmatrix}
    \right\rangle_{\XTraces{j}}^+}
    & = \abs{\sum_{j=1}^2 \left(\left\langle -\laplaceOperator K_j \hat{\phi}_{j,\dataDirichlet} + s \laplaceOperator V_j \hat{\phi}_{j,\dataNeumann}, \overline{\hat{\psi}_{j,\dataNeumann}} \right\rangle_{\Gamma_j}+ \left\langle \frac{1}{s} \laplaceOperator W_j \hat{\phi}_{j,\dataDirichlet} + \laplaceOperator K_j' \hat{\phi}_{j,\dataNeumann}, \overline{\hat{\psi}_{j,\dataDirichlet}} \right\rangle_{\Gamma_{j}} \right)}\\
    & \hspace{-1em} \leq\max_{j\in\{1,2\}} \max\left\{ \norm{s \laplaceOperator V_j}, \norm{\laplaceOperator K_j}, \norm{\laplaceOperator K_j'}, \norm{\frac{1}{s} \laplaceOperator W_j} \right\} \\
    & \times \sum_{j=1}^2 \left(\norm{\hat{\phi}_{j,\dataDirichlet}}_{H^{\frac{1}{2}}\left(\Gamma_j\right)} + \norm{\hat{\phi}_{j,\dataNeumann}}_{H^{-\frac{1}{2}}\left(\Gamma_j\right)}\right) \left(\norm{\hat{\psi}_{j,\dataDirichlet}}_{H^{\frac{1}{2}}\left( \Gamma_j \right)}+\norm{\hat {\psi}_{j,\dataNeumann}}_{H^{-\frac{1}{2}}\left(\Gamma_{j}\right)}\right) \\
    & \hspace{-1em} \overset{\text{(\ref{table_map_prop})}}{\leq} C \max\{\abs{s}^2,\abs{s}\} 2 \norm{\vectorValued{\hat{\phi}}}_{\XMulti} \norm{\vectorValued{\hat{\psi}}}_{\XMulti},
  \end{align*}
where the constant $C$ is the same as in (\ref{table_map_prop}); since $\lvert
s \rvert\geq\sigma_{0}$, taking $\eta=2C \min\left\{  1, 1/\sigma_{0} \right\}
$ the continuity estimate follows. The coercivity directly follows from
Prop.~\ref{PropEll}.
\end{proof}

\begin{remark}
The properties of $\left\langle \boldsymbol{\mathsf{A} } (s) \cdot,
\cdot\right\rangle _{\Sigma}^{+}$ as stated in Lemma~\ref{th:continuousLemma} trivially carry over to its restriction to any subspace of $\XMulti$. 
For our application, the subspace $\XSingleZero \subset \XMulti$ is of particular interest.
\end{remark}

\begin{lemma}
  \label{th:sesquilinearGeneralContinuous}
  The sesquilinear form $\aZero (s) : \XMulti \times \XMulti \to \mathbb{C}$ defined in (\ref{eq:a0}) is continuous: there exists a constant $\eta>0$ independent of $s$ such that
\[
\abs{a_0 \left(s;\vectorValued{\hat \phi}, \vectorValued{\hat \psi} \right)}
\leq\left(  \frac{1}{2} + \eta\max\left\{ \abs{s}^{2}, \abs{s} \right\}
\right)  \norm{\vectorValued{\hat \phi}}_{\XMulti}
\norm{\vectorValued{\hat{\psi}}}_{\XMulti} \quad\forall\boldsymbol{\hat \phi}
,\boldsymbol{\hat \psi} \in\boldsymbol{X} ^{\operatorname{mult}} .
\]
\end{lemma}
\begin{proof}
For the second term in (\ref{eq:a0}) related to
\textquotedblleft$-\frac{\boldsymbol{Id} }{2}$ \textquotedblright{} we get
$\frac{1}{2} \abs{\left\langle
\vectorValued{\hat{\phi}},\overline{\vectorValued{\hat{\psi}}}\right\rangle_{\XMulti}^{+}}
\leq\frac{1}{2}\norm{\vectorValued{\hat{\phi}}}_{\XMulti}%
\norm{\vectorValued{\hat{\psi}}}_{\XMulti}.$ For the term in
(\ref{eq:a0}) related to $\boldsymbol{\mathsf{A} } (s)$, we
use Lemma~\ref{th:continuousLemma}, and the continuity estimate follows.
\end{proof}

Next, we will prove continuity and coercivity of $\aMix(s)$;

\begin{theorem}
  \label{th:mixCoercivity}
The sesquilinear form $\aMix(s)$ is coercive: for the
constant $\zeta>0$ as in Lemma~\ref{th:continuousLemma}, it holds
\[
  \real \aMix \left(s; \boldsymbol{\hat{\phi}} ,\boldsymbol{\hat{\phi}} \right)
  \geq\zeta\frac{\real s}{\abs{s} ^{2}}\norm{\vectorValued{\hat{\phi}}}_{\XMulti}^{2}\quad
  \forall\boldsymbol{\hat{\phi}} \in \XSingleZero , \forall s \in \mathbb{C}_{\sigma_{0}}.
\]
\end{theorem}

\begin{proof}
  From (\ref{eq:a0}), (\ref{th:bilinearEquivalent}) and the definition of $\aImp (s)$ we obtain
  \begin{align*}
    \real \aMix(s;\boldsymbol{\hat{\phi}} ,\boldsymbol{\hat{\phi}}) & = \real \left( \left\langle \boldsymbol{\mathsf{A}}(s) \boldsymbol{\hat{\phi}}, \conjugate{\boldsymbol{\hat{\phi}}} \right\rangle _{\Sigma}^{+} + a^{\operatorname{imp}}(s;\boldsymbol{\hat{\phi}}, \boldsymbol{\hat{\phi}} ) - \frac{1}{2} \left\langle \boldsymbol{\hat{\phi}}, \conjugate{\boldsymbol{\hat{\phi}}} \right\rangle ^{+}_{\Gamma_{\dataImpedance}}\right) \\
                                                                    &  = \real \left(  \left\langle \boldsymbol{\mathsf{A} } (s) \boldsymbol{\hat{\phi}}, \conjugate{\boldsymbol{\hat{\phi}}} \right\rangle _{\Sigma}^{+} + \left\langle \hat{\phi}_{\dataNeumann} - \mathsf{T} (s) \hat{\phi}_{\dataDirichlet}, \overline{\hat{\phi}_{\dataDirichlet}}\right\rangle _{\Gamma_{\dataImpedance}} - \frac{1}{2} \left\langle \boldsymbol{\hat \phi} , \overline{\boldsymbol{\hat \phi} }
                                                                      \right\rangle ^{+}_{\Gamma_{\dataImpedance}} \right) \\
                                                                    &  = \real \left(  \left\langle \boldsymbol{\mathsf{A} } (s) \boldsymbol{\hat{\phi}}, \conjugate{\boldsymbol{\hat{\phi}}} \right\rangle _{\Sigma}^{+} - \left\langle \mathsf{T} (s) \hat{\phi}_{\dataDirichlet}, \overline{\hat{\phi}_{\dataDirichlet}}\right\rangle _{\Gamma_{\dataImpedance}}\right)  .
\end{align*}
We employ Assumption \ref{A1} and Lemma \ref{th:continuousLemma} to obtain%

\begin{equation*}
  \real \aMix(s; \vectorValued{\laplaceFunction{\phi}}, \vectorValued{\laplaceFunction{\phi}}) \geq
  \real \left\langle \vectorValued{\laplaceOperator{A}}(s) \vectorValued{\laplaceFunction{\phi}}, \conjugate{\vectorValued{\laplaceFunction{\phi}}} \right\rangle_{\Sigma}^{+} \geq
  \zeta \frac{\real s}{\abs{s}^{2}} \norm{\vectorValued{\laplaceFunction{\phi}}}_{\XMulti}^{2},
  \forall \vectorValued{\laplaceFunction{\phi}} \in \XSingleZero, s \in \mathbb C_{\sigma_0}.
\end{equation*}
\end{proof}

\begin{theorem}
The sesquilinear form $\aMix(s)$ is continuous: there exists
a constant $\eta>0$ independent of $s \in\mathbb{C}_{\sigma_{0}}$ such that
for all $s \in\mathbb{C}_{\sigma_{0}}$
\[
\abs{\aMix(s;\vectorValued{\laplaceFunction{\phi}},\vectorValued{\laplaceFunction{\psi}})}
\leq\left( \norm{\laplaceOperator T(s)}_{H_{\dataNeumann}^{-1/2}\left(
\Gamma_{\dataImpedance}\right)  \leftarrow H_{\dataDirichlet}^{1/2} \left(
\Gamma_{\dataImpedance}\right) } +\eta\max\left\{ \abs{s}^{2},\abs{s} \right\}
\right)  \norm{\vectorValued{\laplaceFunction{\phi}}}_{\XMulti}
\norm{\vectorValued{\laplaceFunction{\psi}}}_{\XMulti} \quad\forall\vectorValued{\laplaceFunction{\phi}}
,\vectorValued{\laplaceFunction{\psi}} \in \XSingleZero.
\]

\end{theorem}

\begin{proof}
The definition of $\aMix(s)$ implies
\[
\abs{\aMix \left(s;\vectorValued{\hat{\phi}},\vectorValued{\hat{\psi}}\right)}
\leq\abs{a_0 (s;\vectorValued{\hat{\phi}},\vectorValued{\hat{\psi}})} +
\abs{\left\langle \laplaceOperator T(s) \hat{\phi}_{\dataDirichlet}, \overline{\hat{\psi}_{\dataDirichlet}} \right\rangle_{\Gamma_{\dataImpedance}}}.
\]

Lemma~\ref{th:sesquilinearGeneralContinuous} gives an estimate for the first term, while the continuity of the second term follows from the continuity of $\laplaceOperator{T}$:
\begin{equation*}
  \abs{\left\langle \laplaceOperator{T}(s) \laplaceFunction{\phi}_{\dataDirichlet}, \overline{\laplaceFunction{\psi}_{\dataDirichlet}} \right\rangle_{\Gamma_{\dataImpedance}}}
  \leq\norm{\laplaceOperator T(s)}_{H_{\dataNeumann}^{-1/2} \left( \Gamma_{\dataImpedance}\right) \leftarrow H_{\dataDirichlet}^{1/2} \left(\Gamma_{\dataImpedance}\right) } \norm{\vectorValued{\hat \phi}}_{\XMulti} \norm{\vectorValued{\hat \psi}}_{\XMulti}.
\end{equation*}
\end{proof}

Coercivity in time domain is obtained as in \cite[Section~3.7]{sayas_book}:
recall the coercivity estimate:
\begin{equation*}
  \real \aMix \left(s; \vectorValued{\laplaceFunction{\phi}} ,\vectorValued{\laplaceFunction{\phi}} \right) \geq
  \zeta\frac{\real s}{\abs{s} ^{2}} \norm{\vectorValued{\laplaceFunction{\phi}}}_{\XMulti}^{2} \geq \zeta \sigma_0 \norm{s^{-1} \vectorValued{\laplaceFunction{\phi}}}_{\XMulti}^{2},
\end{equation*}
since $\real s \geq \sigma_{0}$;
from
\begin{equation*}
  \real \aMix(s;\vectorValued{\laplaceFunction{\phi}} ,\vectorValued{\laplaceFunction{\phi}}) =
  \real \left(\left\langle \vectorValued{\laplaceOperator{A}}(s) \vectorValued{\laplaceFunction{\phi}}, \conjugate{\vectorValued{\laplaceFunction{\phi}}} \right\rangle_{\Sigma}^{+} -
    \left\langle \laplaceOperator{T}(s) \laplaceFunction{\phi}_{\dataDirichlet}, \overline{\laplaceFunction{\phi}_{\dataDirichlet}} \right\rangle _{\Gamma_{\dataImpedance}}\right),
\end{equation*}
we get that the time domain form of the coercivity estimate is, for $\vectorValued{\timeFunction{\phi}} \in \mathcal C^0(\intervalCO{0,\infty}, \XSingleZero), \sigma_{0}>0$:
\begin{equation*}
  \real \int_0^{\infty} \e^{-2 \sigma_{0} t} \left(\left\langle \vectorValued{\timeOperator{A}}(t) * \vectorValued{\timeFunction{\phi}} ,\conjugate{\vectorValued{\timeFunction{\phi}}} \right\rangle_{\Sigma}^{+} -
    \left\langle \timeOperator{T}(t) * \timeFunction{\phi}_{\dataDirichlet}, \overline{\timeFunction{\phi}_{\dataDirichlet}} \right\rangle _{\Gamma_{\dataImpedance}}\right)  \de t
  \geq \zeta \sigma_{0} \int_{0}^{\infty} \e^{-2 \sigma_{0} t} \norm{\partial_{t}^{-1} \vectorValued{\timeFunction{\phi}}(t)}_{\XMulti}^2.
\end{equation*}
\bibliographystyle{abbrv}
\bibliography{nlailu}
\end{document}